\documentclass[letterpaper,10pt]{amsart}
\usepackage[utf8]{inputenc}
\usepackage[T2A,T1]{fontenc}

\usepackage{url}

\usepackage{amssymb}

\newcommand{\vol}{\operatorname{vol}}
\newcommand{\dif}{{\mathrm d}}
\newcommand{\exterioderivative}{{\mathrm d}}
\newcommand{\exterioderivativex}{{\mathrm d}x}
\newcommand{\trace}{\operatorname{tr}}
\newcommand{\Ext}{\operatorname{Ext}}
\newcommand{\Id}{\operatorname{Id}}
\usepackage{amscd}
\newcommand{\Interpolant}{\mathcal I}

\newcommand{\supp}{\operatorname{supp}}
\newcommand{\grad}{\operatorname{grad}}
\newcommand{\curl}{\operatorname{curl}}

\newcommand{\divergence}{\operatorname{div}}

\makeatletter
\newcommand\suchthat{%
 \@ifstar
  {\mathrel{}\middle|\mathrel{}}
  {\mid}%
}
\makeatother

\DeclareMathOperator*{\esssup}{esssup}

\newtheorem{theorem}{Theorem}[section]

\newtheorem{assumption}[theorem]{Assumption}

\newtheorem{remark}[theorem]{Remark}

\newtheorem{corollary}[theorem]{Corollary}

\begin{document}

\title
[Local Finite Element Approximation of Sobolev Differential Forms]
{Local Finite Element Approximation\\of Sobolev Differential Forms}
\thanks{EG was supported by NSF grants DMS-1703719 and DMS-2012427. 
This material is based upon work supported by the National Science Foundation 
under Grant No. DMS-1439786 while the third author was in residence at the Institute for Computational and 
Experimental Research in Mathematics in Providence, RI, during the ``Advances in Computational Relativity'' program.}

\author{Evan S.\ Gawlik}
\address{Department of Mathematics, University of Hawaii at Manoa}
\email{egawlik@hawaii.edu}

\author{Michael J.\ Holst}
\address{UCSD, Department of Mathematics, La Jolla, California 92093-0112, USA}
\email{mholst@math.ucsd.edu}

\author{Martin W.\ Licht}
\address{ICERM, Brown University, Providence, Rhode Island 02903, USA}
\email{martin\_licht@brown.edu}



\begin{abstract}
 We address fundamental aspects in the approximation theory of vector-valued finite element methods,
 using finite element exterior calculus as a unifying framework. 
 We generalize the Cl\'ement interpolant and the Scott-Zhang interpolant to finite element differential forms,
 and we derive a broken Bramble-Hilbert Lemma.
 Our interpolants require only minimal smoothness assumptions
 and respect partial boundary conditions. 
 This permits us to state local error estimates in terms of the mesh size.  
 Our theoretical results apply to curl-conforming and divergence-conforming finite element methods over simplicial triangulations. 
\end{abstract}

\maketitle

\section{Introduction}

With this article we contribute to an aspect of vector-valued finite element methods 
which has seen increasing interest throughout recent years, 
namely the detailed study of quantitative approximation estimates. 
More specifically, we construct and analyze analogues of the Cl\'ement interpolant and the Scott-Zhang interpolant
for vector-valued finite element methods over simplicial meshes. 
We present our results in the framework of \emph{finite element exterior calculus} (FEEC).

One of the classical results in finite element theory is the quasi-optimality of the finite element solution: 
the Galerkin approximation is just as good as the best approximation, up to a generic constant. 
This is well-known for the vector-valued finite element spaces that have enjoyed popularity in numerical electromagnetism long since.
However, not much is known about the quantitative approximation estimates in terms of the mesh size, in sharp contrast to the scalar-valued setting. 
Only recently have publications started to address this topic in the vector-valued setting; see the literature review further down this introduction. 
The most classical convergence theorem in the finite element analysis for the Poisson problem asserts that the Galerkin error vanishes in the $L^{2}$ norm by the order $\mathcal O(h^{s})$,
where $s > 1$ denotes the Sobolev smoothness of the true solution.
Generally speaking, $s$ can be arbitrarily close to $1$. 
This classical estimate can be proven with the Cl\'ement interpolant. 

As the Cl\'ement interpolant~(see~\cite{clement1975approximation}) is arguably one of the most classical tools in numerical analysis,
the first goal of this article is extending the Cl\'ement interpolant to vector-valued finite element spaces. 
For that purpose we introduce a biorthogonal system of bases and degrees of freedom. 
This is a technical tool of interest on its own. 
If the finite element space is not subject to boundary conditions, 
then the generalization from the scalar-valued case may be regarded as a mere technical note.
This might partially explain why previous publications have not given much attention to this topic.

However, the generalization to finite element differential forms (and thus vector-valued finite elements)
is not quite as trivial as one might think when homogeneous boundary values are imposed. 
As in the scalar-valued case of Cl\'ement's original publication, 
the Cl\'ement interpolant is modified by leaving out the corresponding degrees of freedom along the boundary. 
But while there are numerous tricks in the literature to derive Bramble-Hilbert-type error estimates for Lagrange elements with boundary conditions, 
this is more than a mere technicality in finite element vector calculus. 
Our solution is to reformulate the degrees of freedom as momenta over facets of the triangulation. 
Thus we extend the degrees of freedom to differential forms with only minimal regularity assumptions that allow a meaningful notion of trace:
we assume that both the differential form and its exterior derivative are integrable.
This regularity assumption is natural since it also sufficient to define homogeneous boundary traces in a generalized sense, namely via an integration by parts formula. 

Incidentally, extending the degrees of freedom to differential forms with rough coefficients allows us to generalize another classical concept to finite element exterior calculus: 
we construct a Scott-Zhang-type interpolant. 
The Scott-Zhang interpolant (see~\cite{scott1990finite}) is a local interpolant onto the finite element space which respects homogeneous boundary conditions. 
We replicate that interpolant in finite element exterior calculus. 
Apart from momenta over full-dimensional simplices, the Scott-Zhang-type interpolant also requires integrals along facets.
Thus it is only well-defined for differential forms that allow traces onto facets.

Additionally, the ideas of the Scott-Zhang interpolant have recently been instrumental in proving a broken Bramble-Hilbert lemma for Lagrange elements. 
Using Veeser's exposition~\cite{veeser2016approximating} as a primary source, we show a broken Bramble-Hilbert lemma for finite element differential forms. 
Our Scott-Zhang-type interpolant uses only momenta over full-dimensional cells and facets. 
Prospective applications of this broken Bramble-Hilbert lemma include the convergence theory of finite element exterior calculus over surfaces and manifolds. 
We leave this for future research. The remainder of this introduction provides further context for our research and a partial review of the literature. 
\\

The Hodge-Laplace equation is the central equation in the calculus of differential forms; 
it allows a reformulation as a saddle point problem which is central in finite element exterior calculus. 
The latter formulation captures different formulations of the Maxwell system, and the Poisson problem
in primal and mixed formulation (see Hiptmair~\cite{hiptmair2002finite} and Arnold, Falk, and Winther~\cite{AFW1}). 
It shows that the analytical properties of these partial differential equations 
over a domain $\Omega$ are best described by studying the Sobolev de~Rham complexes
\begin{align}
 \label{intro:sobolevderhamcomplex}
 \begin{CD}
  \dots
  @>\exterioderivative>>
  H\Lambda^k(\Omega)
  @>\exterioderivative>>
  H\Lambda^{k+1}(\Omega)
  @>\exterioderivative>>
  \dots
 \end{CD}
\end{align}
Here, $\exterioderivative$ is the exterior derivative, and a differential $k$-form is in $H\Lambda^k(\Omega)$
if its coefficients are square integrable and its exterior derivative, initially defined in the sense of distributions, has square integrable coefficients as well. 
Specifically, the above example of a Sobolev de~Rham complex is useful for analysing the Hodge-Laplace equation with \emph{natural} boundary conditions. 
The theory of the Hodge-Laplace equation with \emph{essential} or \emph{mixed} boundary conditions has seen substantial progress only in recent years. 
For the Hodge-Laplace equation with mixed boundary conditions we study Sobolev de~Rham complexes with \emph{partial boundary conditions}
(see Gol'dshtein, Mitrea, and Mitrea~\cite{GMM}):
\begin{align}
 \label{intro:sobolevderhamcomplex:mitrand}
 \begin{CD}
  \dots
  @>\exterioderivative>>
  H\Lambda^k(\Omega,\Gamma)
  @>\exterioderivative>>
  H\Lambda^{k+1}(\Omega,\Gamma)
  @>\exterioderivative>>
  \dots
 \end{CD}
\end{align}
Here, \emph{partial boundary condition} refers to imposing homogeneous boundary conditions
along a part $\Gamma \subseteq \partial\Omega$ of the domain boundary. 
The most important results for the de Rham complex 
with either no boundary conditions (when $\Gamma=\emptyset$) 
or full boundary conditions (when $\Gamma=\partial\Omega$),
such as Rellich embedding theorems, Poincar\'e-Friedrichs inequalities, and homology space theory 
are still valid for general mixed boundary conditions 
(see also Jochmann~\cite{jochmann1997compactness,jochmann1999regularity} and Jakab, Mitrea, and Mitrea~\cite{jakab2009regularity}).

In regards to the numerical analysis for the Hodge-Laplace equation, 
\emph{finite element de~Rham complexes} mimic Sobolev de~Rham complexes as a fundamental structure on a discrete level. 
We adopt the framework of \emph{finite element exterior calculus} (Arnold, Falk, and Winther~\cite{AFW1,AFW2}) 
as unifying language for the theoretical background and the formulation of finite element methods.
A very general Galerkin theory of Hilbert complexes, 
which asserts that Galerkin approximations are quasi-optimal approximations of the solution of the Hodge-Laplace equation, 
is at our disposal once we have \emph{smoothed projections} from Sobolev de~Rham complexes onto finite element de~Rham complexes,
that is, $L^2$-bounded projections such that diagrams such as the following commute: 
\begin{align}
 \label{intro:smoothedprojectioncommutingdiagramm}
 \begin{CD}
  \dots
  @>\exterioderivative>>
  H\Lambda^k(\Omega,\Gamma)
  @>\exterioderivative>>
  H\Lambda^{k+1}(\Omega,\Gamma)
  @>\exterioderivative>>
  \dots
  \\
  @. @V{\pi^{k}}VV @V{\pi^{k+1}}VV @.
  \\
  \dots
  @>\exterioderivative>>
  \mathcal P_{r}^{-}\Lambda^{k  }(\mathcal T,\mathcal U)
  @>\exterioderivative>>
  \mathcal P_{r}^{-}\Lambda^{k+1}(\mathcal T,\mathcal U)
  @>\exterioderivative>>
  \dots
 \end{CD}
\end{align}
The widely studied special cases $\Gamma = \emptyset$ and $\Gamma = \partial\Omega$ 
correspond to either 
imposing no essential boundary conditions at all or essential boundary conditions 
along the entire boundary. 
We remark that the cohomology spaces of finite element de~Rham complexes with partial boundary conditions 
were addressed first by Licht~\cite{licht2017complexes} via purely algebraic means,
and Poincar\'e-Friedrichs constants have been addressed 
by Christiansen~and~Licht~\cite{licht2020poincare} within an algebraic framework.

Notably, the concept of commuting bounded projection from Sobolev de~Rham complexes onto finite element de~Rham complexes 
has been the dominating focus of published theoretical research on vector-valued finite element methods. 
Numerous techniques and variations are found in the literature. 
The basic idea, and its relevance to mixed finite element methods,
can at least be traced back to the work of Fortin~\cite{fortin1977analysis} on mixed methods for the Poisson problem. 
Christiansen~\cite{christiansen2007stability} introduced a bounded projection that commutes with the exterior derivative up to a controllable error. 
Arnold, Falk, and Winther~\cite{AFW1} developed a commuting $L^{2}$-bounded projection 
from the de~Rham complex without boundary conditions onto a finite element differential complex assuming quasi-uniform families of triangulations. 
Christiansen and Winther~\cite{christiansen2008smoothed} extended those ideas to the $L^{2}$ de~Rham complex with boundary conditions 
and merely shape-regular families of triangulations. 
Licht described smoothed projections for $L^{p}$ de~Rham complexes over weakly Lipschitz domains, 
first without boundary conditions~\cite{licht2019smoothed} and subsequently with partial boundary conditions~\cite{licht2019mixed};
the existence of such a projection had been stipulated previously by Bonizzoni, Buffa, and Nobile~\cite{bonizzoni2014moment}. 
A commuting bounded local projection was described by Sch\"oberl~\cite{schoberl2008posteriori} in vector-analytic language,
which was later generalized to partial boundary conditions by Gopalakrishnan and Qiu~\cite{gopalakrishnan2011partial},
and to the setting of differential forms by Demlow and Hirani~\cite{demlow2014posteriori}. 
Christiansen, Munthe-Kaas and Owren~\cite{christiansen2011topics} discussed a bounded commuting quasi-projection
that locally preserves polynomials up to specified degree. 
Falk and Winther~\cite{falk2014local} developed a commuting local projection from the $L^{2}$ de~Rham complex without boundary conditions that is bounded in $H\Lambda$-norms. 
Ern and Guermond~\cite{ern2016mollification} described an $L^{p}$-bounded commuting projection in the language of vector analysis. 
One major commonality of these operators is that they provide quasi-optimal approximations within finite element spaces 
while featuring additional properties, such as uniform bounds, commutativity with differential operators, or locality. 
One of the most important applications of these operators has been in proving quasi-optimality of Galerkin approximations in mixed finite element methods~\cite{AFW1}.

However, this does not \emph{quantify} the error of the (quasi)-optimal approximation within the finite element space.  
For example, such would provide error estimates for the finite element solution in terms of powers of the mesh size. 
Hence an additional interpolation error estimate is necessary for that last step.
Numerous results have been published, with most of the work addressing scalar-valued finite element methods only. 
The most widely known interpolation is due to Cl\'ement~\cite{clement1975approximation}.
The Cl\'ement interpolant is local, $L^{p}$-bounded and can be modified to respect homogeneous boundary conditions. 
Another milestone in the literature on quantitative interpolation estimates is the Scott-Zhang interpolant~\cite{scott1990finite}.
This operator interpolates also values over the faces (and thus boundary conditions) and is idempotent, 
however, it generally requires higher smoothness on the function than the Cl\'ement interpolant. 
Surprisingly, only a few publications study quantitative error estimates for vector-valued finite element methods. 
We mention the quasi-optimal interpolant of Ern and Guermond~\cite{ern2017finite} as the apparently first such construction in the literature. 
Their projection operator, which generalizes ideas of Oswald~\cite{oswald1993bpx} to curl- and divergence-conforming finite element spaces, 
satisfies similar local error estimates as the Cl\'ement interpolant and can be modified to satisfy homogeneous boundary conditions.
It seems their publication was the first to give quantitative error estimates
for curl-conforming and divergence-conforming finite element spaces.

Apart from quasi-interpolation error estimates for vector-valued finite element methods,
for which we study the Cl\'ement interpolant and the Scott-Zhang interpolant in finite element exterior calculus, 
we are interested in what has been in circulation as \emph{broken Bramble-Hilbert Lemma} in recent years. 
In the context of finite element methods, the broken Bramble-Hilbert Lemma for scalar-valued functions states 
that approximation by continuous piecewise polynomial functions is essentially as good as approximation by discontinuous piecewise polynomial functions 
(that is, approximation within a \emph{broken} finite element space)
under the condition that the function to be approximated satisfies some moderate continuity conditions. 
This has been investigated by Veeser~\cite{veeser2016approximating} using techniques for the Scott-Zhang interpolant,
see also Camacho and Demlow~\cite{camacho2015L2} for applications to surface finite element methods
and also Bank and Yserentant~\cite{bank2015note} for relations to a~posteriori error estimation. 
Whereas the original proof by Veeser discusses the approximation of $H^{1}$ functions with piecewise higher smoothness, 
we discuss the approximation of differential forms with $H\Lambda$ regularity with piecewise higher smoothness. 
We remark that the projection of Christiansen, Munthe-Kaas and Owren~\cite{christiansen2011topics} 
satisfies a similar result under abstract assumptions. 
The case of divergence-conforming finite element spaces has been addressed by 
Ern, Gudi, Smears, and Vohral{\'\i}k~\cite{ern2019equivalence} with a particular focus on the stability in the polynomial degree. 
A similar result for curl-conforming spaces has been shown by Chaumont-Frelet and Vohral{\'\i}k~\cite{chaumont2020equivalence}.
The aforementioned two contributions, which come closest to the research efforts in this work,
focus on the Hilbert space situation and the perspective on Veeser's result as the \emph{equivalence of global and local approximations} in the $L^{2}$ norm.
We assume the perspective on Veeser's result as a \emph{broken Bramble-Hilbert Lemma}
as in Camacho and Demlow's aforementioned contribution. Like in their case, our result is motivated the error analysis of finite element methods over surfaces and manifolds.
\\

The remainder of this article is structured as follows. 
In Section~\ref{sec:triangulation} we review notions of triangulations. 
In Section~\ref{sec:analysis} we recapitulate basic results about Sobolev differential forms. 
In Section~\ref{sec:femspaces} we review finite element spaces of differential forms. 
Section~\ref{sec:biorthogonal} discusses biorthogonal bases and degrees of freedom. 
Section~\ref{sec:clement} introduces and analyzes the Cl\'ement interpolant for differential forms. 
Section~\ref{sec:neudof} discusses another representation of the degrees of freedom. 
This is used subsequently in Section~\ref{sec:clementboundary},
where we discuss the Cl\'ement interpolant with boundary conditions, 
and in Section~\ref{sec:scottzhang},
where discuss the Scott-Zhang interpolant and the broken Bramble-Hilbert Lemma. 
Finally, Section~\ref{sec:applications} discusses a few applications in the language of vector analysis.

\section{Triangulations} \label{sec:triangulation}
We commence with gathering a few definitions concerning simplices and triangulations. 

Recall that a \emph{simplex} of dimension $d$ is the convex closure of $d+1$ affinely independent points,
which are called the \emph{vertices} of that simplex.
A simplex $F$ is a \emph{subsimplex} of a simplex $T$ if all vertices of $F$ are vertices of $T$.
For any $d$-dimensional simplex $T$ we write $\mathcal F(T)$ for the set of its \emph{facets},
which are the $d+1$ subsimplices of $T$ sharing all but one vertex with $T$. 
More generally, $\Delta_{d}(T)$ is the set of $d$-dimensional simplices of $T$,
and we write $\Delta(T)$ for the set of subsimplices of $T$.

A \emph{simplicial complex} is a collection $\mathcal T$ of simplices that is closed under taking subsimplices
and 
for which the intersection of any two subsimplices $T, T' \in \mathcal T$ 
is either empty or a common subsimplex of $T$ and $T'$.
We say that $\mathcal T$ is $n$-dimensional if
every simplex $T \in \mathcal T$ is a subsimplex of an $n$-dimensional simplex of $\mathcal T$. 
A \emph{simplicial subcomplex} of $\mathcal T$ is any simplicial complex $\mathcal U \subseteq \mathcal T$. 
We write $\Delta_{d}(\mathcal T)$ for the set of $d$-dimensional simplices of $\mathcal T$.

All simplices are assumed to have a fixed orientation. 
Whenever $T$ is a simplex and $F \in \mathcal F(T)$, 
then we set $o(F,T) = 1$ if the orientation of $F$ is induced from $T$ and we set $o(F,T) = -1$ otherwise. 

We introduce another combinatorial condition on the simplicial complex,
discussed in~\cite{veeser2016approximating}. 
We call a finite simplicial complex $\mathcal T$ \emph{face-connected}
whenever 
for all $n$-dimensional simplices $T_0, T \in \mathcal T$ with non-empty intersection, 
there exists a sequence $T_1,\dots,T_N$ of $n$-dimensional simplices of $\mathcal T$
with $T_N = T$, and such that 
for all $1 \leq i \leq N$ we have that $F_{i} = T_{i} \cap T_{i-1}$
satisfies $F_{i} \in \mathcal F(T_{i}) \cap \mathcal F(T_{i-1})$ and $T_0 \cap T \subseteq F_{i}$. 
For example, any simplicial complex that triangulates a domain is face-connected. 

For any simplex $T$ of positive dimension $d$ we let 
$h_T$ and $\vol^{d}(T)$ be its diameter and its $d$-dimensional Hausdorff volume, respectively. 
We call $\mu(T) = h_{T}^{d} / \vol^{d}(T)$ the \emph{shape measure} of $T$. 
The shape measure $\mu(\mathcal T)$ of any simplicial complex $\mathcal T$ is the supremum 
of the shape measures of all its non-vertex simplices. 
Generally speaking, a high shape measure indicates degeneracy of simplices. 
To simplify some technical arguments, 
we write $h_{V}$ for the minimum length of any edge adjacent to some vertex $V \in \mathcal T$. 

For any $T \in \mathcal{T}$ we introduce the two sets 
\begin{gather} \label{math:simplexenvironmenten}
 {U}_{T,\mathcal T}^{\ast} = \bigcup_{ \substack{ T' \in \Delta_{n}(\mathcal{T}) \\ T' \cap T \neq \emptyset } } T',
 \qquad 
 {U}_{T,\mathcal T} = \bigcup_{ \substack{ T' \in \Delta_{n}(\mathcal{T}) \\ T \subseteq T' } } T'. 
\end{gather}
Note that ${U}_{T,\mathcal T} \subseteq {U}_{T,\mathcal T}^{\ast}$.
We remark that the ratio of diameters of adjacent simplices as well as the number of simplices 
entering the unions in \eqref{math:simplexenvironmenten} can be bounded in terms of the shape measure. 

\begin{remark}
 In the analysis of finite element methods, one is commonly interested in results that are valid for families
 of algorithmically constructed triangulations. These triangulations typically satisfy uniform bounds 
 on the mesh constants introduced above.
\end{remark}

\section{Background in Analysis} \label{sec:analysis}

In this section we recapitulate notions and results from the analysis of Sobolev spaces and exterior calculus.
Our focus here are the Sobolev-Slobodeckij spaces, sometimes also referred to as \emph{fractional Sobolev spaces} 
\cite{slobodeckij1958generalized,di2012hitchhiker},
and the calculus of differential forms with coefficients in said Sobolev-Slobodeckij spaces~\cite{scott1995Lp,iwaniec1999nonlinear,mitrea2002traces,GMM}.
Although we are initially only working over domains, 
most notions in this section also apply to the analysis on simplices.

For the remainder of this section, let $\Omega \subseteq \mathbb R^{n} $ be a domain.
\\

We use standard notations for function spaces in this article. 
$C^{\infty}(\Omega)$ is the space of smooth functions over $\Omega$ 
and 
$C^{\infty}(\overline \Omega)$ is the space of restrictions of smooth functions over the Euclidean space onto $\Omega$. 
We write $C_{c}^{\infty}(\Omega)$ for the space of smooth functions with support contained compactly in $\Omega$.
Next, $L^{p}(\Omega)$ is the Lebesgue space over $\Omega$ to the integrability exponent $p \in [1,\infty]$, 
equipped with the norm ${\|\cdot\|}_{L^{p}(\Omega)}$.

Here and in the sequel, 
$A(n)$ is the set of all multiindices over $\{1,\dots,n\}$. 
For any $m \in \mathbb N_0$,  
let $W^{m,p}(\Omega)$ be the \emph{Sobolev space} of measurable functions over $\Omega$ 
for which all distributional $\alpha$-th derivatives with $\alpha \in A(n)$ and $|\alpha| \leq m$ are functions in $L^{p}(\Omega)$.
We recall the norm ${\|\cdot\|}_{W^{m,p}(\Omega)}$ and the seminorm ${|\cdot|}_{W^{m,p}(\Omega)}$,
whose definitions for every $\omega \in W^{m,p}(\Omega)$ are 
\begin{align} \label{math:sobolevnormen}
 \| \omega \|_{W^{m,p}(\Omega)} 
 &:=
 \sum_{ \substack{ \alpha \in A(n) \\ |\alpha| \leq m } } 
 \| \partial^{\alpha} \omega \|_{L^{p}(\Omega)},
 \quad 
 | \omega |_{W^{m,p}(\Omega)} 
 :=
 \sum_{ \substack{ \alpha \in A(n) \\ |\alpha| = m } } 
 \| \partial^{\alpha} \omega \|_{L^{p}(\Omega)}.
\end{align}

In order to define Sobolev-Slobodeckij spaces,
with which one generalizes the idea of the Sobolev space to positive non-integer order,
we let $\theta \in (0,1)$ and define the seminorms
\begin{align*}
 |\omega|_{W^{m+\theta,p}(\Omega)}
 &:=
 \sum_{ \substack{ \alpha \in A(n) \\ |\alpha| = m } }
 \left( \int_{\Omega} \int_{\Omega}
  \dfrac{ | \partial^{\alpha} \omega(x) - \partial^{\alpha} \omega(y) |^{p} }{ | x - y |^{n+p\theta} }
 \dif x \dif y \right)^{\frac 1 p}
 ,
 \quad
 \omega \in W^{m,p}(\Omega), \quad p < \infty,
 \\
 |\omega|_{W^{m+\theta,\infty}(\Omega)}
 &:=
 \sum_{ \substack{ \alpha \in A(n) \\ |\alpha| = m } }
 \esssup\limits_{ (x,y) \in \Omega \times \Omega } 
  \dfrac{ | \partial^{\alpha} \omega(x) - \partial^{\alpha} \omega(y) | }{ | x - y |^{\theta} }
 ,
 \quad
 \omega \in W^{m,p}(\Omega)
 .
\end{align*}
Accordingly, we define the Sobolev-Slobodeckij norm
\begin{align} \label{math:sobolevslobodeckijnormen}
 \|\omega\|_{W^{m+\theta,p}(\Omega)}
 :=
 \|\omega\|_{W^{m,p}(\Omega)}
 +
 |\omega|_{W^{m+\theta,p}(\Omega)}
\end{align}
and let $W^{m+\theta,p}(\Omega)$ denote the Banach space of measurable functions
for which ${\|\cdot\|}_{W^{m+\theta,p}(\Omega)}$ is bounded.
This space is called the \emph{Sobolev-Slobodeckij} space.
\\

We let $C^{\infty}\Lambda^{k}(\Omega)$ and $C^{\infty}\Lambda^{k}(\overline \Omega)$ 
be the spaces of differential $k$-forms with coefficients in $C^{\infty}(\Omega)$ and $C^{\infty}(\overline \Omega)$, respectively.
The space of smooth compactly supported differential forms $C_{c}^{\infty}\Lambda^{k}(\Omega)$ is defined analogously. 
The spaces $L^{p}\Lambda^{k}(\Omega)$ and $W^{s,p}\Lambda^{k}(\Omega)$ are defined accordingly 
for any $p \in [1,\infty]$ and $s \in [0,\infty)$
and one writes ${\|\cdot\|}_{L^{p}\Lambda^{k}(\Omega)}$, ${\|\cdot\|}_{W^{s,p}\Lambda^{k}(\Omega)}$, and ${|\cdot|}_{W^{s,p}\Lambda^{k}(\Omega)}$
for the corresponding norms and seminorms. 

The exterior product $\omega \wedge \eta$ of a $k$-form $\omega$ and an $l$-form $\eta$
is bilinear in each argument and satisfies the identity $\omega \wedge \eta = (-1)^{kl} \eta \wedge \omega$. 
The exterior derivative is a differential operator between differential forms.
One defines 
\begin{align} \label{math:exteriorderivative}
 \exterioderivative\omega = \sum_{i=1}^{n} \exterioderivativex_{i} \wedge \partial_{i} \omega,
 \quad 
 \omega \in C^{\infty}\Lambda^{k}(\Omega).
\end{align}
An important identity is the Leibniz rule 
\begin{align} \label{math:leibnizrule}
 \exterioderivative\left( \omega \wedge \eta \right)
 =
 \exterioderivative\omega \wedge \eta + (-1)^{k} \omega \wedge \exterioderivative\eta,
 \quad 
 \omega \in C^{\infty}\Lambda^{k}(\Omega),
 \quad
 \eta \in C^{\infty}\Lambda^{l}(\Omega).
\end{align}
The exterior derivative of differential forms with coefficients in Lebesgue spaces 
is defined a priori in the sense of distributions. 
A particular class of differential $k$-forms which is of interest in this article is 
\begin{align} \label{math:Wpqspace}
 \mathcal W^{p,q}\Lambda^{k}(\Omega)
 :=
 \left\{\; 
  \omega \in L^{p}\Lambda^{k}(\Omega)
  \suchthat* 
  \exterioderivative\omega \in L^{q}\Lambda^{k+1}(\Omega)
 \;\right\},
 \quad 
 p,q \in [1,\infty].
\end{align}
Our interest in $\mathcal W^{p, q}\Lambda^{k}(\Omega)$ is based on 
that these differential forms, although they have a very low regularity,
allow a meaningful trace theory.
It should be noted that $\mathcal W^{2,2}\Lambda^{k}(\Omega)$ is exactly the Hilbert space $H\Lambda^{k}(\Omega)$,
which is the centre of interest of many publications on finite element exterior calculus.
\\

We remark that, if $S$ is any simplex in $\mathbb R^{n}$ of any dimension $d$, 
one can set up the calculus of differential forms as well,
using the coordinate system of the affine subspace corresponding to $S$. 
We will only need the space $C^{\infty}\Lambda^{k}(S)$ and subspaces of it, 
and leave out the technical details, which are straight-forward. 
We remark that the integral $\int_{S} \omega$ of any integrable $k$-form over a $k$-dimensional simplex $S$ is well-defined.
The trace from any simplex $S$ onto any of its subsimplices $F \in \Delta(S)$ is written $\trace_{S,F}$ in this article.
We also write $\trace_{S}$ for the trace onto any simplex $S$ whenever this well-defined;
there will be no ambiguity in this article regarding this.
\\

We are interested in spaces of differential forms 
that satisfy homogeneous boundary conditions, in a sufficiently generalized sense,
along some subset $\Gamma \subseteq \partial \Omega$ of the domain boundary. 
We refer to such boundary conditions as \emph{partial boundary conditions}.
Our definition of such partial boundary conditions follows Gol'dshtein, Mitrea, and Mitrea (see Definition~3.3 of~\cite{GMM})
and Fernandes and Gilardi~\cite{fernandes1997magnetostatic}, building upon an integration by parts identity. 

Formally, 
when $\Gamma \subseteq \partial \Omega$ is a relatively open subset of $\partial \Omega$, 
then the space $\mathcal W^{p,q}\Lambda^k(\Omega,\Gamma)$ is defined as 
the subspace of $\mathcal W^{p,q}\Lambda^{k}(\Omega)$ whose members adhere to the following condition: 
we have $\omega \in \mathcal W^{p,q}\Lambda^k(\Omega,\Gamma)$ if and only if
for all $x \in \Gamma$ there exists $\rho > 0$ such that
over the open ball $B_{\rho}(x) \subseteq \mathbb R^{n}$ of radius $\rho > 0$ around $x$
we have the identity 
\begin{align} \label{math:partialboundaryconditions}
 \int_{\Omega \cap B_{\rho}(x)} \omega \wedge \exterioderivative \eta
 =
 (-1)^{k+1}
 \int_{\Omega \cap B_{\rho}(x)} {\exterioderivative \omega} \wedge \eta,
 \quad 
 \eta \in C^{\infty}_c\Lambda^{n-k-1}\left(B_{\rho}(x)\right).
\end{align}
One sees immediately that every $\omega \in C^{\infty}(\overline \Omega)$
that vanishes along $\Gamma$ satisfies this identity. 
Formally, this definition of homogeneous boundary values requires no assumptions 
on the regularity of $\partial \Omega$, and thus we circumvent the discussion of traces, 
but of course one has to be careful in which circumstances the general above definition is mathematically helpful. 

One notices that $\mathcal W^{p,q}\Lambda^k(\Omega,\Gamma)$ is a closed subspace of $\mathcal W^{p,q}\Lambda^{k}(\Omega)$. 
We also say that $\omega \in \mathcal W^{p,q}\Lambda^k(\Omega,\Gamma)$ satisfies partial boundary conditions along $\Gamma$.
The definition implies that 
\begin{align}
 \exterioderivative \mathcal W^{p,q}\Lambda^{k}(\Omega,\Gamma) \subseteq \mathcal W^{q,r}\Lambda^{k+1}(\Omega,\Gamma),
 \quad 
 p,q,r \in [1,\infty].
\end{align} 
In other words, 
if a differential form satisfies partial boundary conditions along $\Gamma$,
then its exterior derivative satisfies partial boundary conditions along $\Gamma$, too. 
\\

\begin{remark}
Spaces of differential forms constitute differential complexes which are known as \emph{de~Rham complexes} in the literature. 
For example, writing $H\Lambda^{k}(\Omega,\Gamma) = \mathcal W^{2,2}\Lambda^{k}(\Omega,\Gamma)$, 
consider the differential complex 
\begin{align}
 \label{math:sobolevderhamcomplex}
 \begin{CD}
  \dots
  @>\exterioderivative>>
  H\Lambda^k(\Omega,\Gamma)
  @>\exterioderivative>>
  H\Lambda^{k+1}(\Omega,\Gamma)
  @>\exterioderivative>>
  \dots
 \end{CD}
\end{align}
The case $\Gamma = \emptyset$ corresponds to imposing no boundary conditions at all 
while the case $\Gamma = \partial\Omega$ corresponds to imposing boundary conditions along the whole of the boundary.
The space $H\Lambda^{k}(\Omega,\Gamma)$ is then more commonly written either $H\Lambda^{k}(\Omega)$ or $H_{0}\Lambda^{k}(\Omega)$, respectively. 
Both cases have been subject to extensive study in the literature of theoretical and numerical analysis, 
while results for partial boundary conditions are more recent.
For the case that $\Omega$ is a weakly Lipschitz domain and $\Gamma$ is a boundary part with sufficient regularity,
the images of the exterior derivatives of the de~Rham complex~\eqref{math:sobolevderhamcomplex} 
have closed range and they realize the Betti numbers of $\overline \Omega$ relative to $\Gamma$ on cohomology. 
We refer to~\cite{GMM} for the details.

The study of differential complexes such as \eqref{math:sobolevderhamcomplex}
provides the theoretical background of partial differential equations associated with the exterior derivative. 
The most widely known one is the Hodge-Laplace equation. 
The de~Rham complex with partial boundary conditions 
is the theoretical underpinning for the Hodge-Laplace equation with \emph{mixed boundary conditions} (see~\cite{licht2019mixed}). 
\end{remark}

\section{Finite Element Spaces over Triangulations} \label{sec:femspaces}

We now turn our attention to the theory of finite element differential forms.
We consider the classes of polynomial differential forms, and the corresponding finite element spaces, 
that have been elaborated upon by Hiptmair~\cite{hiptmair2002finite} and Arnold, Falk, and Winther~\cite{AFW1,afwgeodecomp},
\\

We let $\mathcal P_{r}\Lambda^{k}(\Omega)$ be the space of differential $k$-forms
whose coefficients are polynomials of degree at most $r \geq 0$ over the domain $\Omega$. 
For $r \geq 1$, we define $\mathcal P_{r}^{-}\Lambda^{k}(\Omega)$ by 
\begin{align} \label{math:trimmedpolynomialdifferentialforms}
 \mathcal P_{r}^{-}\Lambda^{k}(\Omega) := \mathcal P_{r-1}\Lambda^{k}(\Omega) + \kappa \mathcal P_{r-1}\Lambda^{k+1}(\Omega),
\end{align}
where $\kappa$ is the Koszul operator (see~\cite{AFW1}). 
These spaces are also defined over simplices:  
we let 
$\mathcal P_{r}\Lambda^{k}(S)$ and $\mathcal P_{r}^{-}\Lambda^{k}(S)$
be the pullbacks of the spaces 
$\mathcal P_{r}\Lambda^{k}(\mathbb R^{n})$ and $\mathcal P_{r}^{-}\Lambda^{k}(\mathbb R^{n})$
onto any simplex $S$, respectively. 

We also need to discuss spaces of polynomial differential forms over simplices with boundary conditions.
For any simplex $S$ one sets 
\begin{align} \label{math:spacesonsimpliceswithboundaryconditions}
 \mathring{\mathcal P}_{r}\Lambda^{k}(S) 
 &:=
 \left\{ 
    \omega \in \mathcal P_{r}\Lambda^{k}(S) 
    \suchthat 
    \forall F \in \Delta(S), F \neq S : \trace_{S,F} \omega = 0
 \right\},
 \\
 \mathring{\mathcal P}_{r}^{-}\Lambda^{k}(S) 
 &:=
 \left\{ 
    \omega \in \mathcal P_{r}^{-}\Lambda^{k}(S) 
    \suchthat* 
    \forall F \in \Delta(S), F \neq S : \trace_{S,F} \omega = 0
 \right\}.
\end{align}
We define finite element spaces over triangulations 
by considering piecewise polynomial differential forms satisfying the necessary continuity conditions 
so that the exterior derivative exists not just in the sense of distributions.
Formally, assume that $\mathcal T$ is a triangulation of the domain $\Omega$. 
We set 
\begin{align} \label{math:finiteelementspaces}
 \mathcal P_{r}\Lambda^{k}(\mathcal T)
 &:=
 \left\{ 
    \omega \in \mathcal W^{\infty,\infty}\Lambda^{k}(\Omega)
    \suchthat* 
    \forall T \in \Delta_{n}(\mathcal T) : \omega_{|T} \in \mathcal P_{r}\Lambda^{k}(T)
 \right\}
 ,
 \\
 \mathcal P_{r}^{-}\Lambda^{k}(\mathcal T)
 &:=
 \left\{ 
    \omega \in \mathcal W^{\infty,\infty}\Lambda^{k}(\Omega)
    \suchthat* 
    \forall T \in \Delta_{n}(\mathcal T) : \omega_{|T} \in \mathcal P_{r}^{-}\Lambda^{k}(T)
 \right\}
 .
\end{align}
The definition of finite element spaces with boundary conditions requires further concepts. 
For any simplicial complex $\mathcal U \subseteq \mathcal T$ we define formally 
\begin{gather} \label{math:finiteelementspaces:boundaryconditions}
 \mathcal P_{r}\Lambda^{k}(\mathcal T,\mathcal U)
 :=
 \left\{\; 
  u \in \mathcal P_{r}\Lambda^{k}(\mathcal T)
  \suchthat 
  \forall F \in \mathcal U : \trace_{F} u = 0
 \;\right\}
 ,
 \\
 \mathcal P_{r}^{-}\Lambda^{k}(\mathcal T,\mathcal U)
 :=
 \left\{\; 
  u \in \mathcal P_{r}^{-}\Lambda^{k}(\mathcal T)
  \suchthat* 
  \forall F \in \mathcal U : \trace_{F} u = 0
 \;\right\}
 .
\end{gather}
In the case where $\mathcal U = \emptyset$, we have $\mathcal P\Lambda^{k}(\mathcal T,\mathcal U) = \mathcal P\Lambda^{k}(\mathcal T)$.
Of course, the most interesting case is the setting where $\mathcal U$ triangulates a boundary of a domain 
along which we impose homogeneous partial boundary conditions.

We recapitulate some simple relations between these finite element spaces,
which are easily verifiable from the literature on finite element differential forms: 
\begin{gather*}
   \mathcal P_{r}^{}\Lambda^{k}(\mathcal T,\mathcal U)
   \subseteq 
   \mathcal P_{r+1}^{-}\Lambda^{k}(\mathcal T,\mathcal U)
   \subseteq 
   \mathcal P_{r+1}^{}\Lambda^{k}(\mathcal T,\mathcal U),
   \\
   \exterioderivative \mathcal P_{r+1}\Lambda^{k}(\mathcal T,\mathcal U) = \exterioderivative \mathcal P_{r+1}^{-}\Lambda^{k}(\mathcal T,\mathcal U) \subseteq \mathcal P_{r}\Lambda^{k+1}(\mathcal T,\mathcal U).
\end{gather*}
These hold for any $k,r \in \mathbb Z$ with $r \geq 0$. 

\begin{remark}
 We highlight a few further facts in relation to boundary conditions. 
 Suppose that $\Omega \subseteq \mathbb R^{n}$ is a domain and that $\Gamma \subseteq \partial \Omega$ is some part of its boundary with positive surface measure.
 For the purpose of illustration, let us assume that $\partial\Omega$ can locally be written as the graph of a function. 
 Suppose that the simplicial complex $\mathcal T$ is a triangulation of $\Omega$
 and that the subcomplex $\mathcal U$ is a triangulation of $\Gamma$. One finds 
 \begin{gather*}
  \mathcal P_{r}\Lambda^{k}(\mathcal T,\mathcal U)     = \mathcal P_{r}\Lambda^{k}(\mathcal T) \cap \mathcal{W}^{\infty,\infty}(\Omega,\Gamma)
  ,
  \\
  \mathcal P_{r}^{-}\Lambda^{k}(\mathcal T,\mathcal U) = \mathcal P_{r}^{-}\Lambda^{k}(\mathcal T) \cap \mathcal{W}^{\infty,\infty}(\Omega,\Gamma)
  .
 \end{gather*}
 The spaces $\mathcal P_{r}\Lambda^{k}(\mathcal T,\mathcal U)$ and $\mathcal P_{r}^{-}\Lambda^{k}(\mathcal T,\mathcal U)$ are, in that sense,
 finite element spaces appropriate for discretizing Sobolev spaces of differential forms with boundary conditions along $\mathcal U$. 
\end{remark}

We discuss the geometric decomposition of finite element spaces.
This theoretical framework may be more abstract than what is usually found in introductory finite element expositions 
but it has been very useful in capturing an essential feature of various finite element spaces,
namely association of shape functions and degrees of freedom to cells of the triangulation. 

We assume that for each $F \in \mathcal T$ we have the \emph{extension operators}
\begin{align} \label{math:extensionoperator}
 \Ext_{F,\mathcal T}^{r,k}   : \mathcal P_{r}\Lambda^{k}(F)     \rightarrow \mathcal P_{r}\Lambda^{k}(\mathcal T),
 \quad 
 \Ext_{F,\mathcal T}^{r,k,-} : \mathcal P_{r}^{-}\Lambda^{k}(F) \rightarrow \mathcal P_{r}\Lambda^{k}(\mathcal T).
\end{align}
which have been defined by Arnold, Falk and Winther~\cite{afwgeodecomp}. 
The two critical properties of these extension operators is that they are right-inverses of the traces, 
\begin{align*}
  \trace_{F} \Ext_{F,\mathcal T}^{r,k} = \Id, \quad \trace_{F} \Ext_{F,\mathcal T}^{r,k,-} = \Id,
\end{align*}
and that for all $S \in \mathcal T$ with $F \nsubseteq S$ we have  
\begin{align*}
  \trace_{S} \Ext_{F,\mathcal T}^{r,k} \mathring{\mathcal P}_{r}\Lambda^{k}(F) = 0, \quad \trace_{S} \Ext_{F,\mathcal T}^{r,k,-} \mathring{\mathcal P}_{r}^{-}\Lambda^{k}(F) = 0
  .
\end{align*}
It is then possible to decompose finite element spaces into direct sums 
\begin{align} \label{math:geodecomp}
    \mathcal P_{r}\Lambda^{k}(\mathcal T,\mathcal U)
    =
    \bigoplus_{ \substack{ F \in \mathcal T \\ F \notin \mathcal U } }
    \Ext_{F,\mathcal T}^{r,k} \mathring{\mathcal P}_{r}\Lambda^{k}(F)
    ,
    \quad
    \mathcal P_{r}^{-}\Lambda^{k}(\mathcal T,\mathcal U)
    =
    \bigoplus_{ \substack{ F \in \mathcal T \\ F \notin \mathcal U } }
    \Ext_{F,\mathcal T}^{r,k,-} \mathring{\mathcal P}_{r}^{-}\Lambda^{k}(F)
    . 
\end{align}
This decomposition is an instance of Theorem~4.3 in~\cite{afwgeodecomp}
applied to the finite element spaces $\mathcal P_{r}\Lambda^{k}(\mathcal T)$ and $\mathcal P_{r}^{-}\Lambda^{k}(\mathcal T)$,
see also Theorems~7.3 and 8.3 in the aforementioned publication,
in the case without boundary conditions. For the case with boundary conditions, see~\cite{licht2017priori}.

\begin{remark}
 Informally, \eqref{math:geodecomp} is a decomposition of the global finite element space 
 into localized ``bubble spaces'' associated with the degrees of freedom. 
 For example, if $k=0$, then we are dealing with the classical Lagrange elements. 
 The Lagrange space over $\mathcal T$ is spanned by localized bubble spaces 
 associated to simplices. This includes the standard hat function 
 associated to the vertices, and quadratic bubbles associated to edges,
 and the bubbles associated to full-dimensional simplices. 
 In case $k=n$, we are just dealing with piecewise discontinuous functions 
 whose degrees of freedom are all associated to full-dimensional cells. 
 Another important example is the case $k=n-1$, where we have divergence-conforming finite element spaces. 
 These can be decomposed into ``bubbles'' associated to either full-dimensional cells 
 or faces of codimension one.
\end{remark}

We finish this section with a discussion of the degrees of freedom for these finite element spaces. 
We consider following spaces of functionals. 
When $F \in \mathcal T$ and $m = \dim(F)$, then we define 
\begin{subequations}
\begin{align} \label{math:degreesoffreedoms}
 \mathcal C_r\Lambda^k(F)
 &:=
 \left\{
  \omega \mapsto \int_F \eta \wedge \trace_{F} \omega \in \mathcal W^{\infty,\infty}\Lambda^{k}(\Omega)^{\ast} 
  \suchthat* 
  \eta \in \mathcal P_{r+k-m}^{-}\Lambda^{m-k}(F)
 \right\},
 \\ 
 \mathcal C_r^{-}\Lambda^k(F)
 &:=
 \left\{
  \omega \mapsto \int_F \eta \wedge \trace_{F} \omega \in \mathcal W^{\infty,\infty}\Lambda^{k}(\Omega)^{\ast} 
  \suchthat* 
  \eta \in \mathcal P_{r+k-m-1}\Lambda^{m-k}(F)
 \right\}.
\end{align}
\end{subequations}
These spaces are algebraically isomorphic to $\mathcal P_{r+k-m}^{-}\Lambda^{m-k}(F)$ and $\mathcal P_{r+k-m-1}\Lambda^{m-k}(F)$, respectively,
see~\cite{AFW1}. 
We define those functionals over $\mathcal W^{\infty,\infty}\Lambda^{k}(\Omega)$ since those differential forms have well-defined traces
(see~\cite{gol1982differential}) but this is only of technical relevance. 
If we restrict the functionals in these sets to $\mathring{\mathcal P}_r\Lambda^k(F)$ and $\mathring{\mathcal P}_r^{-}\Lambda^k(F)$, respectively, 
in the obvious sense, then we obtain the full dual spaces of the local finite element spaces with boundary conditions. 
With little effort (see~\cite{licht2017priori}) it is possible to show that 
\begin{align} \label{math:geodecomp:degreesoffreedom}
    \mathcal P_{r}\Lambda^{k}(\mathcal T,\mathcal U)^{\ast}
    =
    \bigoplus_{ \substack{ F \in \mathcal T \\ F \notin \mathcal U } }
    \mathcal C_{r}\Lambda^{k}(F)
    ,
    \quad
    \mathcal P_{r}^{-}\Lambda^{k}(\mathcal T,\mathcal U)^{\ast}
    =
    \bigoplus_{ \substack{ F \in \mathcal T \\ F \notin \mathcal U } }
    \mathcal C_{r}^{-}\Lambda^{k}(F)
    . 
\end{align}

\begin{remark}
The finite element spaces discussed in this article can be put together to form finite element de~Rham complexes,
for example:  
\begin{align}
 \label{math:FEderhamcomplex}
 \begin{CD}
  \dots 
  @>\exterioderivative>>
  \mathcal P_{r}^{-}\Lambda^{k  }(\mathcal T,\mathcal U)
  @>\exterioderivative>>
  \mathcal P_{r}^{-}\Lambda^{k+1}(\mathcal T,\mathcal U)
  @>\exterioderivative>>
  \dots
 \end{CD}
\end{align}
One can construct projections 
$\pi^{k} : \mathcal W^{p,q}\Lambda^{k}(\Omega,\Gamma) \rightarrow \mathcal P_{r}^{-}\Lambda^{k}(\mathcal T,\mathcal U)$ 
from the Sobolev de~Rham complex onto the finite element de~Rham complex 
which commute with the exterior derivative and satisfy $L^{p}$ bounds 
depending only on the polynomial degree and the mesh quality.
\begin{align}
 \label{math:smoothedprojectioncommutingdiagramm}
 \begin{CD}
  \dots
  @>\exterioderivative>>
  \mathcal W^{p,q}\Lambda^{k  }(\Omega,\Gamma)
  @>\exterioderivative>>
  \mathcal W^{q,s}\Lambda^{k+1}(\Omega,\Gamma)
  @>\exterioderivative>>
  \dots
  \\
  @. @V{\pi^{k}}VV @V{\pi^{k+1}}VV @.
  \\
  \dots
  @>\exterioderivative>>
  \mathcal P_{r}^{-}\Lambda^{k  }(\mathcal T,\mathcal U)
  @>\exterioderivative>>
  \mathcal P_{r}^{-}\Lambda^{k+1}(\mathcal T,\mathcal U)
  @>\exterioderivative>>
  \dots
 \end{CD}
\end{align}
This \emph{smoothed projection} is the key to enable the abstract Galerkin theory of Hilbert complexes (see~\cite{AFW2}). 
The finite element solution of the Hodge-Laplace equation 
is a quasi-optimal approximation of the true solution within the finite element space. 
However, those results do not concretize the approximation estimates.
Concretely, we usually want to bound the error in terms of the (local) mesh size and the solution regularity.
The interpolant derived in this article accomplishes that goal.
\end{remark}

\section{Biorthogonal Bases and Degrees of Freedom} \label{sec:biorthogonal}

In this section we discuss biorthogonal systems of bases and degrees of freedom for finite element spaces. 
This will not only provide helpful tools in the discussion of the Cl\'ement interpolant in subsequent sections
but it is also an interesting result in its own right. 
As a particular feature, the bases and degrees of freedom are localized.
We inductively construct the biorthogonal system in a top-down manner: 
the induction starts with cells associated to the highest dimension
and progressively works itself down the simplex dimensions.

\begin{assumption}
For the remainder of this article 
we let $\mathcal T$ be an $n$-dimensional simplicial complex,
and we let $\mathcal U \subseteq \mathcal T$ be a simplicial subcomplex. 
We assume that $\Omega \subseteq \mathbb R^{n}$ is a domain triangulated by $\mathcal T$
and that $\Gamma \subseteq \partial\Omega$ is a part of the domain boundary 
triangulated by $\mathcal U$.
Moreover, we fix $p \in [1,\infty]$, $k \in \mathbb N_{0}$, a polynomial degree $r \in \mathbb N$, and a family of finite element spaces of differential forms. 
Thus we write 
\begin{align*}
 &
 \mathcal P\Lambda^{k}(\mathcal T) = \mathcal P_{r}\Lambda^{k}(\mathcal T), \quad 
 \mathcal P\Lambda^{k}(\mathcal T,\mathcal U) = \mathcal P_{r}\Lambda^{k}(\mathcal T,\mathcal U), 
 \\
 &\qquad
 \text{ and for all $S \in \mathcal T$}:
 \mathcal P\Lambda^{k}(S) = \mathcal P_{r}\Lambda^{k}(S), 
 \; 
 \mathring{\mathcal P}\Lambda^{k}(S) = \mathring{\mathcal P}_{r}\Lambda^{k}(S),
 \text{ and } 
 \mathcal C\Lambda^{k}(S) = \mathcal C_{r}\Lambda^{k}(S), 
 \\
 &\qquad
 \text{ and for all domains $U \subseteq \mathbb R^{n}$}:
 \mathcal P\Lambda^{k}(U) = \mathcal P_{r}\Lambda^{k}(U), 
\end{align*}
or 
\begin{align*}
 &
 \mathcal P\Lambda^{k}(\mathcal T) = \mathcal P_{r}^{-}\Lambda^{k}(\mathcal T), \quad 
 \mathcal P\Lambda^{k}(\mathcal T,\mathcal U) = \mathcal P_{r}^{-}\Lambda^{k}(\mathcal T,\mathcal U), 
 \\
 &\qquad
 \text{ and for all $S \in \mathcal T$}:
 \mathcal P\Lambda^{k}(S) = \mathcal P_{r}^{-}\Lambda^{k}(S), 
 \; 
 \mathring{\mathcal P}\Lambda^{k}(S) = \mathring{\mathcal P}_{r}^{-}\Lambda^{k}(S), 
 \text{ and } 
 \mathcal C\Lambda^{k}(S) = \mathcal C_{r}^{-}\Lambda^{k}(S), 
 \\
 &\qquad
 \text{ and for all domains $U \subseteq \mathbb R^{n}$}:
 \mathcal P\Lambda^{k}(U) = \mathcal P_{r}^{-}\Lambda^{k}(U). 
\end{align*}
We assume\footnote{Recall that $\mathcal P_{r}\Lambda^{0} = \mathcal P_{r}^{-}\Lambda^{0}$ and $\mathcal P_{r-1}\Lambda^{n} = \mathcal P_{r}^{-}\Lambda^{n}$.} that the first option holds if $k=0$ and that the second option holds if $k=n$.
Furthermore, for the reason of exposition, we introduce for every $S \in \mathcal T$ a set of indices 
\begin{align*}
 I(S) := \{ 1, \dots, \dim\mathring{\mathcal P}\Lambda^{k}(S)\}.
\end{align*}
\end{assumption}

We can now state the main result of this section.

\begin{theorem}[Localised Biorthogonal System] \label{prop:biorthogonal}
 There exist bases $\left\{ \phi_{S,i}^{\ast} \right\}_{i \in I(S)}$ of $\mathcal C\Lambda^{k}(S)$ for each $S$,
 and a basis $\left\{ \phi_{S,i} \right\}_{S \in \mathcal T, i \in I(S)}$ of $\mathcal P\Lambda^{k}(\mathcal T)$ such that 
 the following conditions are satisfied for all $S \in \mathcal T$:
 \begin{gather} 
    \forall S' \in \mathcal T, i \in I(S), j \in I(S') : 
    \phi_{S,i}^{\ast}( \phi_{S',j} )
    = 
    \begin{cases}
    1 & \text{ if } S = S', i = j,
    \\
    0 & \text{otherwise.}
    \end{cases} \label{prop:biorthogonal:biorthogonal}
    \\
    \forall S' \in \mathcal T : S \nsubseteq S' \implies \trace_{S'} \phi_{S,i} = 0 \label{prop:biorthogonal:locality}
 \end{gather}
 In addition to that, for all $S, T \in \mathcal T$ with $S \subseteq T$ and $\dim(T) = n$, 
 \begin{gather}
    \| \phi_{S,i} \|_{L^{p}\Lambda^{k}(T)} 
    \leq C_{\rm{A}} h_{S}^{ \frac{n}{p} - k  },  \label{prop:biorthogonal:boundphi}
    \\
    \| \phi^{\ast}_{S,i}( \omega ) \phi_{S,i} \|_{L^{p}\Lambda^{k}(T)} 
    \leq 
    C_{\rm{A}} 
    \| \omega \|_{L^{p}\Lambda^{k}(T)},
    \quad 
    \omega \in \mathcal P_{r}\Lambda^{k}(T). \label{prop:biorthogonal:bound}
 \end{gather}
 Here, $C_{\rm{A}} > 0$ is a constant which only depends on $p$, $n$, the polynomial degree $r$, and $\mu(\mathcal T)$.
\end{theorem}

\begin{remark}
 The degrees of freedom stated in the theorem are just the same as in \eqref{math:degreesoffreedoms}.
 We will construct a new basis of the finite element space 
 from the geometrically decomposed basis via local modifications. 
 Equation~\eqref{prop:biorthogonal:biorthogonal} just states what we understand as biorthogonality,
 and equation~\eqref{prop:biorthogonal:locality} formalises that the basis forms are localized:
 any form associated to the simplex $S$ vanishes on simplices that do not contain $S$. 
 The estimates~\eqref{prop:biorthogonal:boundphi} and \eqref{prop:biorthogonal:bound} follow from scaling arguments. 
\end{remark}

\begin{proof}[Proof of Theorem~\ref{prop:biorthogonal}]
    First, for every simplex $S \in \mathcal T$ we fix a basis $\left\{ \phi_{S,i,0} \right\}_{i \in I(S)}$
    of the local space $\mathring{\mathcal P}\Lambda^{k}(S)$ and a basis $\left\{ \phi_{S,i}^{\ast} \right\}_{i \in I(S)}$
    of the space $\mathcal C\Lambda^{k}(S)$ such that 
    \begin{align*}
    \phi_{S,i}^{\ast}( \phi_{S,j,0} ) = \delta_{ij}, \quad i,j \in I(S).
    \end{align*}
    We can also assume that the differential forms $\phi_{S,i,0}$ and the functionals $\phi_{S,i}^{\ast}$
    are defined via pullback from a reference simplex. 
    Going from there, we inductively build a basis of $\mathcal P\Lambda^{k}(\mathcal T)$ in a top-down fashion. 

    Let $S \in \mathcal T$ be a simplex of dimension $n$.
    We define $\phi_{S,i} \in \mathcal P\Lambda^{k}(\mathcal T)$ by setting $\phi_{S,i|S} := \phi_{S,i,0}$ over $S$ and $\phi_{S,i|T} := 0$ over all other $n$-dimensional simplices $T \in \mathcal T$.
    It then follows that \eqref{prop:biorthogonal:locality} and \eqref{prop:biorthogonal:biorthogonal} hold 
    for all $S \in \Delta_{n}(\mathcal T)$. 
    Since we assume that $\phi_{S,i}$ and $\phi_{S,i}^{\ast}$ are defined via pullback from reference simplices,
    the two inequalities \eqref{prop:biorthogonal:boundphi} and \eqref{prop:biorthogonal:bound} are valid. 
 
    Next, suppose we have defined $\phi_{S,i} \in \mathcal P\Lambda^{k}(\mathcal T)$ 
    for all $S \in \mathcal T$ with $\dim(S) > m$ and $i \in I(S)$
    such that \eqref{prop:biorthogonal:locality} and \eqref{prop:biorthogonal:biorthogonal} hold 
    for all $S \in \mathcal T$ with $\dim(S) > m$.
    For every $S \in \mathcal T$ with $\dim(S) = m$ we then set 
    \begin{align*}
        \phi_{S,i} 
        := 
        \Ext_{S,\mathcal T} \phi_{S,i,0} 
        - 
        \sum_{ \substack{ T \in \mathcal T \\ S \subsetneq T } } 
        \sum_{ l \in I(T) }
        \phi_{T,l}^{\ast}( \Ext_{S,\mathcal T} \phi_{S,i,0} ) \phi_{T,l}
        ,
    \end{align*}
    where $\Ext_{S,\mathcal T} = \Ext^{r,k}_{S,\mathcal T}$ or $\Ext_{S,\mathcal T} = \Ext^{r,k,-}_{S,\mathcal T}$ as defined in Section~\ref{sec:femspaces},
    depending on our choice of finite element space. 
    
    To check that \eqref{prop:biorthogonal:locality} holds, we let $S' \in \mathcal T$ with $S \nsubseteq S'$.
    Then $T \nsubseteq S'$ for all $T \in \mathcal T$ with $S \subseteq T$.
    Therefore the properties of the extension operators and our induction assumptions lead to  
    \begin{align*}
        \trace_{S'} \phi_{S,i} 
        &= 
        \trace_{S'} \Ext_{S,\mathcal T} \phi_{S,i,0} 
        - 
        \sum_{ \substack{ T \in \mathcal T \\ S \subsetneq T } } 
        \sum_{ l \in I(T) }
        \phi_{T,l}^{\ast}( \Ext_{S,\mathcal T} \phi_{S,i,0} ) \trace_{S'} \phi_{T,l}
        =
        0.
    \end{align*}
    Next we prove \eqref{prop:biorthogonal:biorthogonal}. 
    We see that for all $i,j \in I(S)$
    \begin{align*}
        \phi_{S,j}^{\ast}\left( \phi_{S,i} \right)
        =
        \phi_{S,j}^{\ast}\left( \Ext_{S,\mathcal T} \phi_{S,i,0}  \right)
        - 
        \sum_{ \substack{ T \in \mathcal T \\ S \subsetneq T \\ l \in I(T) } } 
        \phi_{T,l}^{\ast}( \Ext_{S,\mathcal T} \phi_{S,i,0} ) \phi_{S,j}^{\ast}\left( \phi_{T,l} \right)
        &=
        \phi_{S,j}^{\ast}\left( \Ext_{S,\mathcal T} \phi_{S,i,0} \right) 
        \\&=
        \phi_{S,j}^{\ast}\left( \phi_{S,i,0} \right) 
        =
        \delta_{ij}
        .
    \end{align*}
    Let $i \in I(S)$. 
    If $S' \in \mathcal T$ with $S \neq S'$ and $S \nsubseteq S'$,
    then we already know that $\trace_{S'} \phi_{S,i} = 0$, thus $\phi_{S',j}^{\ast}( \phi_{S,i}) = 0$
    for all $j \in I(S')$. 
    If instead $S \subseteq S'$ 
    then for all $j \in I(S')$ one sees 
    \begin{align*}
     \phi_{S',j}^{\ast}\left( \phi_{S,i} \right)
     &=
     \phi_{S',j}^{\ast}\left( \Ext_{S,\mathcal T} \phi_{S,i,0}  \right)
     - 
     \sum_{ \substack{ T \in \mathcal T \\ S \subsetneq T } } 
     \sum_{ l \in I(T) }
     \phi_{T,l}^{\ast}( \Ext_{S,\mathcal T} \phi_{S,i,0} ) \phi_{S',j}^{\ast}\left( \phi_{T,l} \right)
     \\&=
     \phi_{S',j}^{\ast}\left( \Ext_{S,\mathcal T} \phi_{S,i,0}  \right)
     - 
     \phi_{S',j}^{\ast}( \Ext_{S,\mathcal T} \phi_{S,i,0} ) \phi_{S',j}^{\ast}\left( \phi_{S',j} \right)
     =
     0.
    \end{align*}
    Lastly, we attend to the inequalities \eqref{prop:biorthogonal:bound} and \eqref{prop:biorthogonal:boundphi}. 
    In what follows, 
    we write $C$ for a generic positive constant which depends on the same quantities as $C_{\mathrm{A}}$ in the statement of the theorem
    and which may change from line to line.
    By the induction assumption, 
    they are true for simplices $T \in \mathcal T$ with $\dim(T) > \dim(S)$. 
    For any simplex $D \in \mathcal T$ of dimension $n$ with $S \subseteq D$, 
    \begin{align*}
        \| \phi_{S,i} \|_{L^{p}\Lambda^{k}(D)}
        &\leq  
        \| \Ext_{S,\mathcal T} \phi_{S,i,0}  \|_{L^{p}\Lambda^{k}(D)}
        + 
        \sum_{ \substack{ T \in \mathcal T \\ S \subsetneq T } } 
        \sum_{ l \in I(T) }
        \| \phi_{T,l}^{\ast}( \Ext_{S,\mathcal T} \phi_{S,i,0} ) \phi_{T,l} \|_{L^{p}\Lambda^{k}(D)}
        \\&\leq  
        \| \Ext_{S,\mathcal T} \phi_{S,i,0}  \|_{L^{p}\Lambda^{k}(D)}
        + 
        C
        \sum_{ \substack{ T \in \mathcal T \\ S \subsetneq T } } 
        \sum_{ l \in I(T) }
        \| \Ext_{S,\mathcal T} \phi_{S,i,0} \|_{L^{p}\Lambda^{k}(D)}
        .
    \end{align*}
    Our choice of extension operators $\Ext_{S,\mathcal T}$ can be defined equivalently via transformation from a reference simplex,
    and so a scaling argument gives 
    \begin{align*}
     \| \Ext_{S,\mathcal T} \phi_{S,i,0} \|_{L^{p}\Lambda^{k}(D)} \leq C h_{S}^{ \frac{n}{p} - k }
     .
    \end{align*}
    This shows \eqref{prop:biorthogonal:boundphi}. 
    Finally, \eqref{prop:biorthogonal:bound} follows from 
    \begin{align*}
     \| \phi^{\ast}_{S,i}( \omega ) \phi_{S,i} \|_{L^{p}\Lambda^{k}(T)}
     &\leq 
     \left| \phi^{\ast}_{S,i}( \omega ) \right| 
     \| \phi_{S,i} \|_{L^{p}\Lambda^{k}(T)} 
     \\&\leq 
     C h_{S}^{n-1 - (n-1-k)}
     \left\| \omega \right\|_{L^{\infty}\Lambda^{k}(T)}  
     C_{\mathrm{A}} h_{S}^{ \frac{n}{p} - k }
     \\&\leq 
     C h_{S}^{k}
     \left\| \omega \right\|_{L^{\infty}\Lambda^{k}(T)}  
     C_{\mathrm{A}} h_{S}^{ \frac{n}{p} - k }
     \leq 
     C
     h_{S}^{k}
     h_{S}^{ -\frac{n}{p} }
     h_{S}^{ \frac{n}{p} - k }
     \left\| \omega \right\|_{L^{p}\Lambda^{k}(T)}  
    \end{align*}
    where we use another scaling argument and an inverse inequality.
    This completes the induction step,
    and the theorem follows.
\end{proof}

It is easy to extend the preceding theorem to the case of finite element spaces with boundary conditions. 
We simply use only those shape forms which are not associated with simplices of the respective boundary part.

\begin{theorem} \label{prop:biorthogonal:boundary}
    Let $\{ \phi_{S,i} \}_{ S \in \mathcal T, i \in I(S) }$ be the basis of $\mathcal P\Lambda^{k}(\mathcal T)$
    as described in Theorem~\ref{prop:biorthogonal}. 
    Then the set $\{ \phi_{S,i} \}_{ S \in \mathcal T \setminus \mathcal U, i \in I(S) }$ is a basis of $\mathcal P\Lambda^{k}(\mathcal T,\mathcal U)$.
\end{theorem}

\begin{proof}
    Let $\omega \in \mathcal P\Lambda^{k}(\mathcal T,\mathcal U)$.
    There exist unique $\chi_{S,i} \in \mathbb R$ for $S \in \mathcal T$ and $i \in I(S)$ with 
    \begin{align*}
     \omega = \sum_{ S \in \mathcal T } \sum_{ i \in I(S) } \chi_{S,i} \phi_{S,i}. 
    \end{align*}
    It remains to show that $\chi_{S,i} = 0$ for any $S \in \mathcal U$. We use an induction argument. 
    First, if $S \in \mathcal U$ with $\dim(S) = 0$, then $\trace_{S} \omega = \sum_{ i \in I(S) } \chi_{S,i} \trace_{S} \phi_{S,i}$.
    Hence $\trace_{S} \omega = 0$ shows that $\chi_{S,i} = 0$ for all $i \in I(S)$. 
    Next, suppose that for some $m > 0$ we already know that $\chi_{S,i} = 0$ for $S \in \mathcal T$ and $i \in I(S)$ with $\dim(S) < m$. 
    If $S \in \mathcal U$ with $\dim(S) = m$, then property \eqref{prop:biorthogonal:locality} yields 
    \begin{align*}
     0 = \trace_{S} \omega 
     &= 
     \sum_{ \substack{ F \in \mathcal T \\ F \subseteq S } } \sum_{ i \in I(F) } \chi_{F,i} \trace_{S} \phi_{F,i}
     =
     \sum_{ i \in I(S) } \chi_{S,i} \trace_{S} \phi_{S,i}
     .
    \end{align*}
    It follows again $\chi_{S,i} = 0$ for all $i \in I(S)$. 
    An induction argument completes the proof. 
\end{proof}

\section{Cl\'ement Interpolation and Local Approximation Theory} \label{sec:clement}

In this section we generalize the Cl\'ement interpolant without boundary conditions to the setting of differential forms. 
Thus we construct a bounded operator from $L^{p}$ spaces of differential forms onto finite element spaces. 
Our construction follows the main ideas of what is known as the Cl\'ement interpolant in the scalar-valued finite element setting. 
In his original work, Cl\'ement defined the interpolant first taking projections onto local neighborhoods
of the degrees of freedom and then evaluating each degree of freedom at the associated projection.
The resulting operator is bounded with respect to Lebesgue norms, it is local, 
and allows for best approximations in the local neighborhood around each cell.
\\

First, we fix projections onto polynomial differential forms over simplices and neighborhoods of simplices.
For each each simplex $S \in \mathcal T$ we have an idempotent bounded linear mapping 
\begin{align*}
 P_{S} : 
 L^{p}\Lambda^{k}(\Omega)
 \rightarrow 
 \mathcal P_{r}\Lambda^{k}({U}_{S,\mathcal T}) 
 \subset 
 L^{p}\Lambda^{k}(\Omega)
\end{align*}
such that for all $\omega \in W^{m,p}\Lambda^{k}(\Omega)$ with $m \in [0,r+1]$ one has 
\begin{gather}
 \| \omega - P_{S} \omega \|_{L^{p}\Lambda^{k}({U}_{S,\mathcal T})} 
 \leq 
 C_{\rm{BH}} 
 h_{S}^{m}
 | \omega |_{W^{m,p}\Lambda^{k}({U}_{S,\mathcal T})},
 \label{math:bramblehilbert}
\end{gather}
and whenever $\exterioderivative\omega \in W^{l,p}\Lambda^{k}(\Omega)$ with $l \in [0,r]$, 
\begin{gather}
 \| \exterioderivative\omega - \exterioderivative P_{S} \omega \|_{L^{p}\Lambda^{k+1}({U}_{S,\mathcal T})} 
 \leq 
 C_{\rm{BH}} 
 h_{S}^{l}
 | \exterioderivative \omega |_{W^{l,p}\Lambda^{k+1}({U}_{S,\mathcal T})}.
 \label{math:bramblehilbert:exteriorderivative}
\end{gather}
Here, $C_{\rm{BH}} > 0$ depends only on $n$, $p$, the polynomial degree $r$, and the triangulation regularity.
One possible choice for $P_{S}$ is the interpolant introduced by Dupont and Scott~\cite{dupont1980polynomial},
which commutes with partial derivatives. 
While they discuss that mapping only for scalar functions, 
it can easily be extended to differential forms by componentwise application. 

We introduce another family of projections. For each $n$-dimensional simplex $T \in \mathcal T$
there exists a bounded projection 
\begin{align*}
 \Pi_{T} : L^{p}\Lambda^{k}(T) \rightarrow \mathcal P\Lambda^{k}(T)
\end{align*}
which satisfies the inequalities 
\begin{gather}
 \| \omega - \Pi_{T} \omega \|_{L^{p}\Lambda^{k}(T)} 
 \leq 
 C_{\Pi} 
 \inf_{ \psi \in \mathcal P\Lambda^{k}(T) }
 \| \omega - \psi \|_{L^{p}\Lambda^{k}(T)}, 
 \quad 
 \omega \in L^{p}\Lambda^{k}(T),
 \label{math:bestapproximation}
 \\
 \| \exterioderivative\omega - \exterioderivative \Pi_{T} \omega \|_{L^{p}\Lambda^{k+1}(T)} 
 \leq 
 C_{\Pi} 
 \inf_{ \psi \in \mathcal P\Lambda^{k}(T) }
 \| \exterioderivative\omega - \exterioderivative\psi \|_{L^{p}\Lambda^{k+1}(T)}, 
 \quad 
 \omega \in \mathcal W^{p,p}\Lambda^{k}(T). \label{math:bestapproximation:exteriorderivative}
\end{gather}
Here, $C_{\Pi} > 0$ depends only on $n$, $p$, the polynomial degree $r$, and the triangulation regularity.
To see this, we first define the projection on a reference simplex and then transport them to other simplices via pullback. 
On a reference simplex, we simply pick the well-known smoothed projection without boundary conditions 
(see~\cite{AFW1,christiansen2008smoothed,licht2019smoothed,licht2019mixed}). 
These operators are uniformly bounded, commute with the exterior derivative, 
and satisfy \eqref{math:bestapproximation} and \eqref{math:bestapproximation:exteriorderivative} over a reference simplex. 
The desired properties then follow. 
\\

We define our interpolant by 
\begin{align} \label{math:clementinterpolante}
 \Interpolant_{\mathcal T}: L^{p}\Lambda^{k}(\Omega) \rightarrow \mathcal P\Lambda^{k}(\mathcal T),
 \quad 
 \omega \mapsto 
 \sum_{S \in \mathcal T} \sum_{i \in I(S)} 
 \phi^\ast_{S,i} \left( P_{S} \omega \right) \phi_{S,i}.
\end{align}
This generalizes the Cl\'ement interpolant to the setting of finite element exterior calculus. 
Next we analyze the interpolation error.

\begin{theorem} \label{prop:clementerrorestimate}
 There exists $C_{\mathcal{I}} > 0$, depending only on $n$, $p$, the polynomial degree $r$, and the shape measure of the triangulation,
 such that the following is true:
 for all $T \in \Delta_{n}(\mathcal T)$ we have 
 \begin{align*}
    \| \Interpolant_{\mathcal T}\omega \|_{L^{p}\Lambda^{k}(T)}
    &\leq 
    C_{\mathcal{I}}
    \| \omega \|_{L^{p}\Lambda^{k}({U}_{T,\mathcal T}^{\ast})},
    \quad 
    \omega \in L^{p}\Lambda^{k}({U}_{T,\mathcal T}^{\ast}),
 \end{align*}
 and for all $T \in \Delta_{n}(\mathcal T)$ we have 
 \begin{align*}
    \| \omega - \Interpolant_{\mathcal T}\omega \|_{L^{p}\Lambda^{k}(T)}
    \leq 
    C_{\mathcal{I}}
    \| \omega - \Pi_{T} \omega \|_{L^{p}\Lambda^{k}(T)}
    +
    C_{\mathcal{I}}
    \sum_{ \substack{ S \subseteq T \\ i \in I(S) } } 
    \| \omega - P_{S} \omega \|_{L^{p}\Lambda^{k}({U}_{S,\mathcal T})},
    \quad 
    \omega \in L^{p}\Lambda^{k}(\Omega).
 \end{align*}
\end{theorem}

\begin{proof}
    Let $\omega \in L^{p}\Lambda^{k}(\Omega)$. 
    Let $T \in \mathcal T$ be any $n$-dimensional simplex of the triangulation. 
    We estimate 
    \begin{align*}
     \| \Interpolant_{\mathcal T}\omega \|_{L^{p}\Lambda^{k}(T)}
     &\leq 
     \sum_{S \subseteq T} \sum_{i \in I(S)} 
     \| \phi^\ast_{S,i} \left( P_T\omega \right)\phi_{S,i|T} \|_{L^{p}\Lambda^{k}(T)}
     \\&\leq 
     \sum_{S \subseteq T} \sum_{i \in I(S)} 
     C_{\mathrm{A}}
     \| P_T\omega \|_{L^{p}\Lambda^{k}(T)}
     \leq 
     \sum_{S \subseteq T} \sum_{i \in I(S)} 
     C_{\mathrm{A}}
     (1+C_{\mathrm{BH}})
     \| \omega \|_{L^{p}\Lambda^{k}({U}_{S,\mathcal T})}
     .
    \end{align*}
    The first inequality follows from this. 
    
    One notices that  
    \begin{align*}
    \Pi_{T}\omega =  \sum_{S \subseteq T} \sum_{i \in I(S)} \phi^\ast_{S,i} \left( \Pi_T\omega \right) \phi_{S,i|T}.
    \end{align*}
    The difference $\omega - \Interpolant_{\mathcal T}\omega$ over the simplex $T$ can now be rewritten:
    \begin{align*}
    ( \omega - \Interpolant_{\mathcal T}\omega )_{|T}
    &= 
    ( \omega - \Pi_T\omega + \Pi_T\omega - \Interpolant_{\mathcal T}\omega )_{|T}
    \\&=
    ( \omega - \Pi_T\omega )_{|T}
    + 
    \sum_{S \subseteq T} \sum_{i \in I(S)} \left(
    \phi^\ast_{S,i} \left( \Pi_T\omega \right) \phi_{S,i} - \phi^\ast_{S,i} \left( P_{S} \omega_{|T} \right) \phi_{S,i} 
    \right)_{|T}
    \\&=
    ( \omega - \Pi_T\omega )_{|T}
    + 
    \sum_{S \subseteq T} \sum_{i \in I(S)} 
    \phi^\ast_{S,i} \left( \Pi_T\omega - (P_{S} \omega)_{|T} \right) \phi_{S,i|T}
    . 
    \end{align*}
    Therefore it follows that  
    \begin{align*}
    &
    \| \omega - \Interpolant_{\mathcal T}\omega \|_{L^{p}\Lambda^{k}(T)}
    \leq 
    \| \omega - \Pi_T\omega \|_{L^{p}\Lambda^{k}(T)}
    + 
    \sum_{S \subseteq T} \sum_{i \in I(S)} 
    \| \phi^\ast_{S,i} \left( \Pi_T\omega - (P_{S} \omega)_{|T} \right) \phi_{S,i|T} \|_{L^{p}\Lambda^{k}(T)}
    . 
    \end{align*}
    From inequality \eqref{prop:biorthogonal:bound}, we get  
    for each subsimplex $S \subseteq T$ and index $i \in I(S)$ the estimate 
    \begin{align*}
    \| \phi^\ast_{S,i} \left( \Pi_T\omega - (P_{S} \omega)_{|T} \right) \phi_{S,i|T} \|_{L^{p}\Lambda^{k}(T)}
    &
    \leq 
    C_{\rm A} \| \Pi_T\omega - (P_{S} \omega)_{|T} \|_{L^{p}\Lambda^{k}(T)} 
    \\&\leq 
    C_{\rm A} \| \Pi_T\omega - \omega \|_{L^{p}\Lambda^{k}(T)} 
    + 
    C_{\rm A} \| \omega - (P_{S} \omega)_{|T} \|_{L^{p}\Lambda^{k}(T)}
    \\&\leq 
    C_{\rm A} \| \Pi_T\omega - \omega \|_{L^{p}\Lambda^{k}(T)} 
    + 
    C_{\rm A} \| \omega - P_{S} \omega \|_{L^{p}\Lambda^{k}({U}_{S,\mathcal T})}
    .
    \end{align*}
    With some constant $C_{0}$ which depends only on $n$ and the polynomial degree $r$, 
    one can summarize our observations then with the local estimate 
    \begin{align*}
    &
    \| \omega - \Interpolant_{\mathcal T}\omega \|_{L^{p}\Lambda^{k}(T)}
    \leq 
    (1+C_{0}C_{A})
    \| \omega - \Pi_{T} \omega \|_{L^{p}\Lambda^{k}(T)}
    +
    C_{\rm A} 
    \sum_{S \subseteq T} \sum_{i \in I(S)} 
    \| \omega - P_{S} \omega \|_{L^{p}\Lambda^{k}({U}_{S,\mathcal T})}
    .
    \end{align*}
    The desired theorem follows. 
\end{proof}

\begin{corollary}
  Let
  $m \in [0,r+1]$ if $\mathcal P\Lambda^{k}(\mathcal T) = \mathcal P_{r}^{ }\Lambda^{k}(\mathcal T)$
  and let 
  $m \in [0,r  ]$ otherwise. 
  Then for all $T \in \Delta_{n}(\mathcal T)$
  we have 
  \begin{align*}
    \| \omega - \Interpolant_{\mathcal T}\omega \|_{L^{p}\Lambda^{k}(T)}
    &\leq 
    C_{\mathcal{I},0}
    h_{T}^{m} 
    | \omega |_{W^{m,p}\Lambda^{k}({U}_{T,\mathcal T}^{\ast})},
    \quad 
    \omega \in W^{m,p}\Lambda^{k}(\Omega).
  \end{align*}
  Here, $C_{\mathcal{I},0} > 0$ depends only on $n$, $p$, the polynomial degree $r$, and the shape measure of the triangulation.
\end{corollary}

\begin{proof}
 This follows from Theorem~\ref{prop:clementerrorestimate}, Inequalities \eqref{math:bramblehilbert} and \eqref{math:bestapproximation},
 together with standard approximation estimates and the local finiteness of the triangulation. 
\end{proof}

This generalizes the Cl\'ement interpolant to the setting of finite element exterior calculus. 
In particular, we reproduce the order of approximation in the mesh size known from the scalar-valued theory. 
However, the reader will notice that we have only covered the case when no boundary conditions are imposed on the finite element space. 
The generalization to homogeneous boundary conditions,
either along the whole of the boundary or merely a part of it, is not yet covered by this construction. 
Indeed, the interpolant of this section does not preserve homogeneous boundary traces. 

The most obvious modification of the interpolant is simply setting all degrees of freedom along the boundary part to zero, which is also the approach followed in Cl\'ement's original paper~\cite{clement1975approximation}. 
While that straight-forward modification will eventually provide the desired result, 
it is not straight-forward how the best approximation properties can be proven under that modification. 
The next section will prepare technical tools to accomplish that target.

\section{Extending the Degrees of Freedom} \label{sec:neudof}

In order to advance our analysis of finite element interpolation, 
we need to rewrite degrees of freedom in a manner that defines them over differential forms 
with minimal smoothness assumptions.
The idea is that every degree of freedom associated to lower-dimensional simplices 
can be expressed in terms of traces over facets. 

\begin{theorem} \label{prop:extendeddof}
    For every $S,F \in \mathcal T$ with $\dim(F) = n - 1$ and $S \subseteq F$
    and every $i \in I(S)$ 
    there exists 
    $\mathring\xi_{F,S,i} \in C^{\infty}_{c}\Lambda^{n-k-1}(F)$
    such that 
    \begin{align} \label{prop:extendeddof:algebra}
        \int_{F} \mathring\xi_{F,S,i} \wedge \trace_{F} \omega
        = 
        \phi_{S,i}^{\ast}( \omega ),
        \quad 
        \omega \in \mathcal P_{r}\Lambda^{k}(T).
    \end{align}
    Furthermore, for every $T \in \Delta_{n}(\mathcal T)$ with $F \subseteq T$ 
    there exists $\Xi_{T,F,S,i} \in C^{\infty}\Lambda^{n-k-1}(T)$ such that 
    $\trace_{F} \Xi_{T,F,S,i} = \mathring\xi_{F,S,i}$
    and the support of $\Xi_{T,F,S,i}$ has positive distance from all facets of $T$ except $F$. 
    
    Moreover, there exists $C_{\Xi} > 0$, depending only on $n$, $p \in [1,\infty]$, the polynomial degree $r$, and the shape measure of the triangulation, such that 
    \begin{align} \label{prop:extendeddof:scaling}
        \| \Xi_{T,F,S,i} \|_{L^{p}\Lambda^{n-k-1}(T)} 
        \leq 
        C_{\Xi} 
        h_{S}^{ \frac{n}{p} - n+k+1},
        \quad  
        \| \exterioderivative\Xi_{T,F,S,i} \|_{L^{p}\Lambda^{n-k}(T)} 
        \leq 
        C_{\Xi} 
        h_{S}^{ \frac{n}{p} - n+k}.
    \end{align}
\end{theorem}

\begin{proof}
    As to simplify the exposition, 
    this proof is to be read as a continuation of the proof of Theorem~\ref{prop:biorthogonal},
    and we tacitly assume all technical details made in that proof. 
    
    For every $S \in \Delta(F)$ and $i \in I(S)$ we let $\phi^{F}_{S,i} := \trace_{F} \phi_{S,i}$. 
    So $\left\{ \phi^{F}_{S,i} \right\}_{S \in \Delta(F), i \in I(S)}$ is a basis of $\mathcal P\Lambda^{k}(F)$.
    If $\mathcal P\Lambda^{k}(F) \neq \mathcal P_{r}\Lambda^{k}(F)$, we augment to a basis of $\mathcal P_{r}\Lambda^{k}(F)$
    by including differential forms that are first defined on a reference facet and then transported to $F$;
    we write $\mathcal A\Lambda^{k}(F)$ for the resulting basis of $\mathcal P_{r}\Lambda^{k}(F)$. 
    Note that $\mathcal A\Lambda^{k}(F)$ can be defined uniformly via transport from a reference facet. 
    
    One can find a set $\mathcal B\Lambda^{n-k-1}(F) \subset C^{\infty}_{c}\Lambda^{n-k-1}(F)$
    whose members represent the dual basis of $\mathcal A\Lambda^{k}(F)$ by integration over $F$;
    this construction can be done on a reference facet first and then be transported to $F$.
    Since the degrees of freedom are defined via transport from a reference simplex as well,
    one can build $\mathring\xi_{F,S,i} \in C^{\infty}_{c}\Lambda^{n-k-1}(F)$ as desired
    by a linear combination of members of $\mathcal B\Lambda^{n-k-1}(F)$. 
    
    Having constructed $\mathring\xi_{F,S,i} \in C^{\infty}_{c}\Lambda^{n-k-1}(F)$ satisfying \eqref{prop:extendeddof:algebra},
    one easily constructs $\Xi_{T,F,S,i} \in C^{\infty}\Lambda^{n-k-1}(T)$ 
    satisfying $\trace_{F} \Xi_{T,F,S,i} = \mathring\xi_{F,S,i}$
    and such that $\supp \Xi_{T,F,S,i}$ has positive distance from all facets of $T$ except $F$.  
    The existence of a constant $C_{\Xi} > 0$ satisfying \eqref{prop:extendeddof:scaling} follows easily
    from a scaling argument. 
\end{proof}

Any simplex $S \in \mathcal T$ is generally contained in different faces and full-dimensional simplices of the triangulation. 
For technical reasons, for any simplex $S \in \mathcal T$ of dimension at most $n-1$ 
we fix an arbitrary face $F_{S} \in \mathcal T$ with $S \subseteq F$ and a $n$-dimensional simplex $T_{S} \in \mathcal T$ with $F_{S} \subseteq T_{S}$.
We also introduce the abbreviations 
\begin{align} \label{math:festsetzen}
 \mathring\xi_{S,i} := \mathring\xi_{F_{S},S,i}, \quad \Xi_{S,i} := \Xi_{T_{S},F_{S},S,i}.
\end{align}
However, we make one modification if $S \in \mathcal U$: in that case, we require additionally that $F_{S} \in \mathcal U$.
This enforces that degrees of freedom associated to the boundary part $\Gamma$ depend on values over facets within that boundary part.

\begin{remark}
 We make generous use of the following identity. 
 For any $S, F, T \in \mathcal T$ with $T \in \Delta_{n}(\mathcal T)$, $F \in \mathcal F(T)$, $S \subseteq F$,
 and all $i \in I(S)$, 
 the differential forms $\Xi_{T,F,S,i}$ and $\mathring\xi_{F,S,i}$
 satisfy 
 \begin{align*}
    o(F,T) 
    \int_{F} \mathring\xi_{F,S,i} \wedge \trace_{F} \omega
    = 
    \int_{T} \exterioderivative\Xi_{T,F,S,i} \wedge \omega + (-1)^{n-k-1} \Xi_{T,F,S,i} \wedge \exterioderivative\omega,
    \quad 
    \omega \in C^{\infty}\Lambda^{k}(T).
 \end{align*}
 The significance of that formula is the right-hand side substitutes the left-hand side 
 in lieu of a notion of traces if $\omega$'s coefficients are very rough functions. 
 The right-hand side is well-defined even if, say, $\omega \in H\Lambda^{k}(T)$ 
 or more generally $\omega \in \mathcal W^{p,q}\Lambda^{k}(T)$ for any $p,q \in [1,\infty]$. 
\end{remark}

\section{Local Approximation Theory with Partial Boundary Conditions} \label{sec:clementboundary}

We define the modified Cl\'ement interpolant by 
\begin{align}
    \Interpolant_{\mathcal T,\mathcal U}: L^{p}\Lambda^{k}(\Omega) \rightarrow \mathcal P\Lambda^{k}(\mathcal T,\mathcal U),
    \quad 
    \omega \mapsto 
    \sum_{ \substack{ S \in \mathcal T \\ S \notin \mathcal U } } \sum_{i \in I(S)} 
    \phi^\ast_{S,i} \left( P_{S} \omega \right) \phi_{S,i}.
\end{align}
It is evident that $\Interpolant_{\mathcal T,\mathcal U}$ takes values in the finite element space $\mathcal P\Lambda^{k}(\mathcal T,\mathcal U)$
with homogeneous boundary conditions along the boundary part $\Gamma$. 
With the tools from the preceding section, one can prove error estimates.

\begin{theorem} \label{prop:clementerrorestimate:boundary}
    There exists $C_{\mathcal{I},\mathcal U} > 0$, depending only on $n$, $p$, the polynomial degree $r$, and the shape measure of the triangulation,
    such that the following is true:
    for all $T \in \Delta_{n}(\mathcal T)$, 
    \begin{align*}
        \| \Interpolant_{\mathcal T,\mathcal U}\omega \|_{L^{p}\Lambda^{k}(T)}
        &\leq 
        C_{\mathcal{I,\mathcal U}}
        \| \omega \|_{L^{p}\Lambda^{k}({U}_{T,\mathcal T}^{\ast})},
        \quad 
        \omega \in L^{p}\Lambda^{k}(\Omega),
    \end{align*}
    and for all $T \in \Delta_{n}(\mathcal T)$ one has
    \begin{align*}
        \| \omega - \Interpolant_{\mathcal T,\mathcal U}\omega \|_{L^{p}\Lambda^{k}(T)}
        &\leq 
        \| \omega - \Interpolant_{\mathcal T}\omega \|_{L^{p}\Lambda^{k}(T)}
        \\&\qquad
        +
        C_{\mathcal{I,\mathcal U}}
        \sum_{ \substack{ S \subseteq T \\ S \in \mathcal U } } \sum_{i \in I(S)} 
        \| \omega - P_{S} \omega \|_{L^{p}\Lambda^{k}( {U}_{S,\mathcal T} )}
        +
        h_{S} \| \exterioderivative\omega - \exterioderivative P_{S} \omega \|_{L^{p}\Lambda^{k+1}( {U}_{S,\mathcal T} )}
    \end{align*}
    whenever $\omega \in \mathcal W^{p,p}\Lambda^{k}(\Omega,\Gamma)$.
\end{theorem}

\begin{proof}
    Let $T \in \mathcal T$ be any $n$-dimensional simplex. 
    If $T$ has no subsimplex in $\mathcal U$,
    then 
    \begin{align*}
        \left( \Interpolant_{\mathcal T,\mathcal U} \omega \right)_{|T}
        = 
        \left( \Interpolant_{\mathcal T} \omega \right)_{|T},
        \quad 
        \omega \in L^{p}\Lambda^{k}(\Omega),
    \end{align*}
    and one can simply apply Theorem~\ref{prop:clementerrorestimate}. 

    Let us assume instead that $T \in \mathcal T$ 
    is an $n$-dimensional simplex which has a subsimplex contained in $\mathcal U$. 
    Then the first inequality follows similarly as in the proof of Theorem~\ref{prop:clementerrorestimate},
    so we only need to study the second inequality. 
    Obviously,
    \begin{align*}
        \omega_{|T} - \left( \Interpolant_{\mathcal T,\mathcal U} \omega \right)_{|T} 
        =
        \omega_{|T} - \left( \Interpolant_{\mathcal T} \omega \right)_{|T} 
        +
        \sum_{ \substack{ S \subseteq T \\ S \in \mathcal U } } \sum_{i \in I(S)} 
        \phi^\ast_{S,i} \left( (P_{S} \omega)_{|T} \right) \phi_{S,i|T}
    \end{align*}
    Now recall the identity
    \begin{align*}
        \phi^\ast_{S,i} \left( (P_{S} \omega)_{|T} \right)
        =
        \int_{F_{S}} \mathring\xi_{S,i} \wedge \trace_{T_{S},F_{S}} (P_{S} \omega)_{|T_{S}},
    \end{align*}
    which is valid because $(P_{S} \omega)_{|T} \in \mathcal P_{r}\Lambda^{k}(T)$. Now, 
    \begin{gather*}
        \int_{F_{S}} \mathring\xi_{S,i} \wedge \trace_{T_{S},F_{S}} (P_{S} \omega)_{|T_{S}}
        =
        o( F_{S}, T_{S} )
        \int_{T_{S}} \exterioderivative\Xi_{S,i} \wedge (P_{S} \omega)_{|T_{S}} + (-1)^{n-k-1} \Xi_{S,i} \wedge \exterioderivative(P_{S} \omega)_{|T_{S}}
    \end{gather*}
    and since $\omega$ satisfies partial boundary conditions along the boundary part $\Gamma$
    and $F_{S} \subseteq \overline \Gamma$, we get 
    \begin{align*}
        &
        \int_{T_{S}} \exterioderivative\Xi_{S,i} \wedge (P_{S} \omega)_{|T_{S}} + (-1)^{n-k-1} \Xi_{S,i} \wedge \exterioderivative(P_{S} \omega)_{|T_{S}}
        \\&\qquad
        =
        \int_{T_{S}}
        \exterioderivative\Xi_{S,i} \wedge \left( (P_{S} \omega)_{|T_{S}} - \omega \right)
        + 
        \int_{T_{S}}
        (-1)^{n-k-1} \Xi_{S,i} \wedge \exterioderivative\left( (P_{S} \omega)_{|T_{S}} - \omega \right)
        .
    \end{align*}
    Thus, letting 
    ${q} = p/(p-1) \in [1,\infty]$, 
    we use the integration by parts formula and H\"older's inequality to find 
    \begin{align*}
        \left| \phi^\ast_{S,i} \left( (P_{S} \omega)_{|T} \right) \right|
        &\leq 
        \| \exterioderivative\Xi_{S,i} \|_{L^{{q}}\Lambda^{n-k}(T_{S})} 
        \| P_{S} \omega - \omega \|_{L^{p}\Lambda^{k}(T_{S})}
        \\&\qquad+ 
        \| \Xi_{S,i} \|_{L^{{q}}\Lambda^{n-k-1}(T_{S})} 
        \| \exterioderivative P_{S} \omega - \exterioderivative\omega \|_{L^{p}\Lambda^{k+1}(T_{S})}
        .
    \end{align*}
    Hence we find that 
    \begin{align*}
        &
        \| \omega - \Interpolant_{\mathcal T,\mathcal U} \omega \|_{L^{p}\Lambda^{k}(T)}
        \\&\qquad\leq
        \| \omega - \Interpolant_{\mathcal T} \omega \|_{L^{p}\Lambda^{k}(T)}
        +
        \sum_{ \substack{ S \subseteq T \\ S \in \mathcal U } } \sum_{i \in I(S)} 
        \left| \phi^\ast_{S,i} \left( (P_{S} \omega)_{|T} \right) \right| 
        \| \phi_{S,i} \|_{L^{p}\Lambda^{k}(T)}
        \\&\qquad\leq
        \| \omega - \Interpolant_{\mathcal T} \omega \|_{L^{p}\Lambda^{k}(T)}
        \\&\qquad\qquad+
        \sum_{ \substack{ S \subseteq T \\ S \in \mathcal U } } \sum_{i \in I(S)} 
        \| \exterioderivative\Xi_{S,i} \|_{L^{{q}}\Lambda^{n-k}(T_{S})}
        \| \omega - P_{S} \omega \|_{L^{p}\Lambda^{k}(T_{S})}
        \| \phi_{S,i} \|_{L^{p}\Lambda^{k}(T)}
        \\
        &\qquad\qquad+
        \sum_{ \substack{ S \subseteq T \\ S \in \mathcal U } } \sum_{i \in I(S)} 
        \| \Xi_{S,i} \|_{L^{{q}}\Lambda^{n-k-1}(T_{S})}
        \| \exterioderivative \omega  - \exterioderivative P_{S} \omega \|_{L^{p}\Lambda^{k+1}(T_{S})}
        \| \phi_{S,i} \|_{L^{p}\Lambda^{k}(T)}
    \end{align*}
    We recall the bounds 
    \begin{gather*}
        \| \phi_{S,i} \|_{L^{p}\Lambda^{k}(T)}
        \leq 
        C_{\mathrm{A}} h_{S}^{ \frac{n}{p} - k }
        \\
        \| \Xi_{S,i} \|_{L^{{q}}\Lambda^{n-k-1}(T_{S})} 
        \leq 
        C_{\Xi} 
        h_{S}^{ \frac{n}{{q}} - n+k+1},
        \quad 
        \| \exterioderivative\Xi_{S,i} \|_{L^{{q}}\Lambda^{n-k}(T_{S})} 
        \leq 
        C_{\Xi} 
        h_{S}^{ \frac{n}{{q}} - n+k}.
    \end{gather*}
    Since $1 = 1/p + 1/{q}$, putting this together produces the desired inequality. 
\end{proof}

\begin{corollary}
  Let
  $m \in [0,r+1]$ if $\mathcal P\Lambda^{k}(\mathcal T) = \mathcal P_{r}^{ }\Lambda^{k}(\mathcal T)$
  and let 
  $m \in [0,r  ]$ otherwise. 
  Write $l := \max(0,m-1)$. 
  Then for all $T \in \Delta_{n}(\mathcal T)$ and all $\omega \in \mathcal W^{p,p}\Lambda^{k}(\Omega) \cap W^{m,p}\Lambda^{k}(\Omega,\Gamma)$
  one has 
  \begin{align*}
    \| \omega - \Interpolant_{\mathcal T,\mathcal U}\omega \|_{L^{p}\Lambda^{k}(T)}
    &\leq 
    C_{\mathcal{I},\mathcal U,0}
    \left( 
        h_{T}^{m} 
        | \omega |_{W^{m,p}\Lambda^{k  }({U}_{T,\mathcal T}^{\ast})}
        +
        h_{T}^{l+1} 
        | \exterioderivative\omega |_{W^{l,p}\Lambda^{k+1}({U}_{T,\mathcal T}^{\ast})}
    \right)
    .
  \end{align*}
  Here, $C_{\mathcal{I},\mathcal U,0} > 0$ depends only on $n$, $p$, the polynomial degree $r$, and the shape measure of the triangulation.
\end{corollary}

\begin{proof}
 We observe $\exterioderivative \omega \in W^{l,p}\Lambda^{k+1}(\Omega)$ for $\omega \in \mathcal W^{p,p}\Lambda^{k}(\Omega) \cap W^{m,p}\Lambda^{k}(\Omega)$. 
 The results follows by combining Theorem~\ref{prop:clementerrorestimate:boundary}, 
 Inequalities \eqref{math:bramblehilbert}--\eqref{math:bestapproximation:exteriorderivative},
 and standard estimates. 
\end{proof}

\section{A Scott-Zhang-type Interpolant} \label{sec:scottzhang}

The Cl\'ement operator, with or without boundary conditions, has only minimal regularity assumptions on its argument:
the operator is bounded over differential forms whose coefficients are in an $L^{p}$ space.
Approximation estimates in terms of the mesh size follow from additional regularity of the interpolated differential form. 

However, the quantitative estimates for the Cl\'ement operator in either variation require smoothness 
of the interpolated differential form over patches of cells, across cell boundaries. 
The Scott-Zhang interpolation for functions in $W^{1,2}(\Omega)$ overcomes this restriction 
and yields approximation error estimates of the same order as the Cl\'ement interpolant 
but merely requiring piecewise higher smoothness. 
One consequence is that continuous Lagrange elements have approximation capability 
equivalent to discontinuous Lagrange elements provided the function has square-integrable first derivatives. 
Furthermore, the Scott-Zhang interpolant preserves homogeneous partial boundary conditions.
In this section we generalize the Scott-Zhang interpolant and the error estimate to the setting of differential forms. 
\\

When $S \in \mathcal T$ with $\dim(S) = n$, then we introduce the mapping 
\begin{align*}
    K_{S,i} : L^{p}\Lambda^{k}(\Omega) \rightarrow \mathbb R, \quad \omega \mapsto \phi^{\ast}_{S,i}(\omega).
\end{align*}
By the choice of degrees of freedom in Section~\ref{sec:femspaces}, these functionals are defined via integration against a smooth differential form over $S$,
and hence they are well-defined even for integrable differential forms. 

If instead $S \in \mathcal T$ with $\dim(S) < n$, then we consider the mapping 
\begin{align*}
    K_{S,i} : \mathcal W^{p,p}\Lambda^{k}(\Omega) \rightarrow \mathbb R, 
    \quad
    \omega 
    \mapsto 
    \int_{T_{S}} \exterioderivative \Xi_{T_{S},F_{S},S,i} \wedge \omega + (-1)^{n-k+1} \Xi_{T_{S},F_{S},S,i} \wedge \exterioderivative\omega
    .
\end{align*}
We define the Scott-Zhang interpolant
\begin{align}
    \mathcal J_{\mathcal T} : \mathcal W^{p,p}\Lambda^{k}(\Omega) \rightarrow \mathcal P\Lambda^{k}(\mathcal T),
    \quad 
    \omega
    \mapsto 
    \sum_{ S \in \mathcal T } \sum_{ i \in I(S) } 
    K_{S,i}( \omega )
    \phi_{S,i}
    .
\end{align}
This completes the construction of our Scott-Zhang-type interpolant. 
We observe that $\omega \in \mathcal W^{p,p}\Lambda^{k}(\Omega,\Gamma)$ implies $K_{S,i}( \omega ) = 0$ whenever $S \in \mathcal U$.
Therefore we have got a mapping 
\begin{align*}
    \mathcal J_{\mathcal T} : \mathcal W^{p,p}\Lambda^{k}(\Omega,\Gamma) \rightarrow \mathcal P\Lambda^{k}(\mathcal T,\mathcal U).
\end{align*}
Next we discuss an error estimate for this approximation operator. 

\begin{theorem} \label{prop:scottzhangestimate}
    There exists $C_{\mathcal{J},\mathcal U} > 0$, depending only on $n$, $p$, the polynomial degree $r$, and the shape measure of the triangulation,
    such that the following is true:
    for all $T \in \Delta_{n}(\mathcal T)$ we have 
    \begin{align*}
        \| \mathcal J_{\mathcal T}\omega \|_{L^{p}\Lambda^{k}(T)}
        \leq 
        C_{\mathcal{J},\mathcal U}
        \sum_{ \substack{ T' \in \Delta_{n}(\mathcal T) \\ T \cap T' \neq \emptyset } }
        \| \omega \|_{ \mathcal W^{p,p}\Lambda^{k}( T' ) },
        \quad 
        \omega \in \mathcal W^{p,p}\Lambda^{k}( \Omega ),
    \end{align*}
    and for all $T \in \Delta_{n}(\mathcal T)$ we have 
    \begin{align*}
        \| \omega - \mathcal J_{\mathcal T}\omega \|_{L^{p}\Lambda^{k}(T)}
        \leq 
        C_{\mathcal{J},\mathcal U}
        \sum_{ \substack{ T' \in \Delta_{n}(\mathcal T) \\ T \cap T' \neq \emptyset } }
            \| \omega - \Pi_{T'} \omega \|_{L^{p}\Lambda^{k}(T')} 
            + 
            h_{T'} 
            \| \exterioderivative \omega - \exterioderivative \Pi_{T'} \omega \|_{L^{p}\Lambda^{k+1}(T')} 
    \end{align*}
    whenever $\omega \in \mathcal W^{p,p}\Lambda^{k}(\Omega,\Gamma)$.
\end{theorem}

\begin{proof}
    The first inequality is easily seen, so we focus on the second inequality. 
    Let $T \in \mathcal T$ be any $n$-dimensional simplex. 
    We find that 
    \begin{align*}
        \omega_{|T} - ( \mathcal J_{\mathcal T} \omega )_{|T} 
        &= 
        \omega_{|T} 
        - 
        \Pi_{T}\omega
        +
        \Pi_{T}\omega
        -
        \sum_{ S \subseteq T } \sum_{ i \in I(S) } 
        K_{S,i}( \omega )
        \phi_{S,i|T}
        \\
        &= 
        \omega_{|T} 
        - 
        \Pi_{T}\omega
        +
        \sum_{ S \subseteq T } \sum_{ i \in I(S) } 
        \phi^\ast_{S,i}( \Pi_T\omega )
        \phi_{S,i|T}
        -
        \sum_{ S \subseteq T } \sum_{ i \in I(S) } 
        K_{S,i}( \omega )
        \phi_{S,i|T}
        .
    \end{align*}
    Hence 
    \begin{align*}
        \| \omega - \mathcal J_{\mathcal T} \omega \|_{L^{p}\Lambda^{k}(T)}
        &\leq  
        \| \omega - \Pi_{T}\omega \|_{L^{p}\Lambda^{k}(T)}
        +
        \sum_{ S \subseteq T } \sum_{ i \in I(S) } 
        \left| \phi^\ast_{S,i}( \Pi_T\omega ) - K_{S,i}( \omega ) \right|
        \| \phi_{S,i} \|_{L^{p}\Lambda^{k}(T)}
        \\&\leq  
        \| \omega - \Pi_{T}\omega \|_{L^{p}\Lambda^{k}(T)}
        +
        \sum_{ S \subseteq T } \sum_{ i \in I(S) } 
        \left| \phi^\ast_{S,i}( \Pi_T\omega ) - K_{S,i}( \omega ) \right|
        C_{\mathrm{A}} h_{S}^{\frac{n}{p}-k} 
        .
    \end{align*}
    We study the terms in the second sum in closer detail. 
    The functionals $\phi^\ast_{T,i}$ and $K_{T,i}$ are the same and thus 
    \begin{align*}
        \phi^\ast_{T,i}( \Pi_T\omega ) - K_{T,i}( \omega )
        &=
        \phi^\ast_{T,i}( \Pi_T\omega - \omega_{|T} )
        .
    \end{align*}
    With H\"older's inequality, a scaling argument and Theorem~\ref{prop:biorthogonal} we thus get the upper bound 
    \begin{align*}
        \left| \phi^\ast_{T,i}( \Pi_T\omega ) - K_{T,i}( \omega ) \right|
        \| \phi_{S,i} \|_{L^{p}\Lambda^{k}(T)}
        &\leq 
        C 
        h_{T}^{ \frac{n(p-1)}{p} - n+k } 
        h_{T}^{ \frac{n}{p} - k } 
        \| \Pi_{T}\omega - \omega \|_{L^{p}\Lambda^{k}(T)}
        .
    \end{align*}
    We dedicate our attention to the degrees of freedom that are associated to proper subsimplices $S$ of $T$. 
    Here, the functionals $\phi^\ast_{S,i}$ and $K_{S,i}$ generally differ. 
    We recall that for any $F \in \mathcal F(\mathcal T)$ with $S \subseteq F$ we have 
    \begin{align*}
        \phi^\ast_{S,i}( \Pi_T\omega ) 
        = 
        \int_{F} \mathring\xi_{F,S,i} \wedge \trace_{T,F} \Pi_{T}\omega 
        = 
        \int_{T} \exterioderivative\Xi_{T,F,S,i} \wedge \Pi_{T}\omega + (-1)^{n-k+1}\Xi_{T,F,S,i} \wedge \exterioderivative \Pi_{T}\omega
        .
    \end{align*}
    On the other hand, 
    \begin{align*}
        K_{S,i}( \omega ) 
        &= 
        \int_{T_{S}} \exterioderivative\Xi_{T_{S},F_{S},S,i} \wedge \omega_{T} + (-1)^{n-k+1}\Xi_{T_{S},F_{S},S,i} \wedge \exterioderivative \omega_{T}
        .
    \end{align*}
    The simplicial complex $\mathcal T$ is face-connected since it triangulates a domain. 
    Therefore there exists a sequence $T_0, T_1, \dots, T_N$ of $n$-dimensional simplices of $\mathcal T$
    without repetitions 
    such that $T_0 = T_{S}$ and $T_N = T$ and such that 
    for all $1 \leq j \leq N$ there exist facets $F_{j} := T_{j} \cap T_{j-1}$
    for which $S \subseteq F_{j}$. 
    Write $F_{0} := F_{S}$ and $F_{N+1} := F$. 
    We utilize the technique of telescope sum and find 
    \begin{align*}
        &
        \phi^\ast_{S,i}( \Pi_{T}\omega )
        -
        K_{S,i}( \omega ) 
        \\&= 
        \int_{F} \mathring\xi_{F,S,i} \wedge \trace_{T,F} \Pi_{T} \omega 
        -
        K_{S,i}( \omega ) 
        \\&= 
        \int_{F_{N+1}} \mathring\xi_{F_{N+1},S,i} \wedge \trace_{T_{N},F_{N+1}} \Pi_{T_{N}} \omega 
        +
        \sum_{j=0}^{N} \phi^\ast_{S,i}( \Pi_{T_{j}}\omega ) - \phi^\ast_{S,i}( \Pi_{T_{j}}\omega )  
        -
        K_{S,i}( \omega ) 
        \\&= 
        \int_{F_{N+1}} \mathring\xi_{F_{N+1},S,i} \wedge \trace_{T_{N},F_{N+1}} \Pi_{T_{N}} \omega 
        -
        \sum_{j=0}^{N} \int_{F_{j+1}} \mathring\xi_{F_{j+1},S,i} \wedge \trace_{T_j,F_{j+1}} \Pi_{T_j} \omega 
        \\&\qquad\qquad\qquad\qquad\qquad\qquad\quad
        +
        \sum_{j=0}^{N} \int_{F_{j  }} \mathring\xi_{F_{j  },S,i} \wedge \trace_{T_j,F_{j  }} \Pi_{T_j} \omega 
        -
        K_{S,i}( \omega ) 
        \\&= 
        \sum_{j=1}^{N} 
        \int_{F_{j}} \mathring\xi_{F_{j},S,i} \wedge \trace_{T_{j  },F_{j  }} \Pi_{T_{j  }} \omega 
        -
        \int_{F_{j}} \mathring\xi_{F_{j},S,i} \wedge \trace_{T_{j-1},F_{j  }} \Pi_{T_{j-1}} \omega 
        \\&\qquad\qquad\qquad\qquad\qquad\qquad\qquad+
        \int_{F_{0}} \mathring\xi_{F_{0},S,i} \wedge \trace_{T_{0},F_{0  }} \Pi_{T_{0}} \omega 
        -
        K_{S,i}( \omega ) 
    \end{align*}
    From the definition of $K_{S,j}$ we get 
    \begin{align*}
        K_{S,i}( \omega ) 
        = 
        \int_{T_{0}} \exterioderivative \Xi_{T_{0},F_{0},S,i} \wedge \omega + (-1)^{n-k+1} \Xi_{T_{0},F_{0},S,i} \wedge \exterioderivative\omega
        .
    \end{align*}
    Aside from that, we know 
    \begin{align*}
        \int_{F_{0}} \mathring\xi_{F_{0},S,i} \wedge \trace_{T_{0},F_{0  }} \Pi_{T_{0}} \omega 
        = 
        \int_{T_{0}} \exterioderivative \Xi_{T_{0},F_{0},S,i} \wedge \Pi_{T_{0}} \omega + (-1)^{n-k+1} \Xi_{T_{0},F_{0},S,i} \wedge \exterioderivative \Pi_{T_{0}} \omega
        .  
    \end{align*}
    Thus it becomes apparent that 
    \begin{align*}
        &
        \int_{F_{0}} \mathring\xi_{F_{0},S,i} \wedge \trace_{T_{0},F_{0  }} \Pi_{T_{0}} \omega 
        -
        K_{S,i}( \omega ) 
        \\&\quad 
        =
        \int_{T_{0}} \exterioderivative \Xi_{T_{0},F_{0},S,i} \wedge \left( \Pi_{T_{0}} \omega - \omega \right) 
        + 
        (-1)^{n-k+1} 
        \Xi_{T_{0},F_{0},S,i} \wedge \exterioderivative \left( \Pi_{T_{0}} \omega - \omega \right)
        .
    \end{align*}
    Therefore, writing $q := p/(p-1)$, 
    \begin{align*}
        &\quad 
        |
        \phi^\ast_{S,i}( \Pi_{T_{0}}\omega )
        -
        K_{S,i}( \omega ) 
        |
        \\&\leq 
        \| \exterioderivative \Xi_{T_{0},F_{0},S,i} \|_{L^{q}\Lambda^{n-k}(T_{0})}
        \| \omega - \Pi_{T_{0}} \omega \|_{L^{p}\Lambda^{k}(T_{0})}
        + 
        \| \Xi_{T_{0},F_{0},S,i} \|_{L^{q}\Lambda^{n-k-1}(T_{0})}
        \| \exterioderivative\omega - \exterioderivative \Pi_{T_{0}} \omega \|_{L^{p}\Lambda^{k+1}(T_{0})}
        \\&\leq 
        C_{\Xi} h_{S}^{ \frac{n(p-1)}{p} - n + k}
        \| \omega - \Pi_{T_{0}} \omega \|_{L^{p}\Lambda^{k}(T_{0})}
        + 
        C_{\Xi} h_{S}^{ \frac{n(p-1)}{p} - n + k + 1}
        \| \exterioderivative\omega - \exterioderivative \Pi_{T_{0}} \omega \|_{L^{p}\Lambda^{k+1}(T_{0})}
        .
    \end{align*}
    Now consider any $1 \leq j \leq N$. 
    By the equivalence of the boundary integrals with an integration by parts formula we find 
    \begin{align*}
        &
        \int_{F_{j}} \mathring\xi_{F_{j},S,i} \wedge \trace_{T_{j  },F_{j  }} \Pi_{T_{j  }} \omega 
        -
        \int_{F_{j}} \mathring\xi_{F_{j},S,i} \wedge \trace_{T_{j-1},F_{j  }} \Pi_{T_{j-1}} \omega 
        \\
        &
        \quad=
        o(F_{j},T_{j})
        \int_{T_{j}} 
        \exterioderivative \Xi_{T_{j},F_{j},S,i} \wedge \Pi_{T_{j}} \omega 
        + 
        (-1)^{n-k+1} 
        \Xi_{T_{j},F_{j},S,i} \wedge \exterioderivative \Pi_{T_{j}} \omega 
        \\&\quad\quad-
        o(F_{j},T_{j-1})
        \int_{T_{j-1}} 
        \exterioderivative \Xi_{T_{j-1},F_{j},S,i} \wedge \Pi_{T_{j-1}} \omega 
        + 
        (-1)^{n-k+1} 
        \Xi_{T_{j-1},F_{j},S,i} \wedge \exterioderivative \Pi_{T_{j-1}} \omega 
        .
    \end{align*}
    Let $\Xi_{F_{j},S,i} \in L^{\infty}\Lambda^{n-k-1}(\Omega)$ 
    with $\Xi_{F_{j},S,i|T_{j}} = \Xi_{T_{j},F_{j},S,i}$ and $\Xi_{F_{j},S,i|T_{j-1}} = \Xi_{T_{j-1},F_{j},S,i}$
    and vanishing on all other $n$-simplices of $\mathcal T$. 
    One sees that $\Xi_{F_{j},S,i} \in \mathcal W^{\infty,\infty}\Lambda^{n-k-1}(\Omega)$ with support in the interior of $T_{j} \cup T_{j-1}$. 
    So an integration by parts reveals that 
    \begin{align*}
        \int_{T_{j} \cup T_{j-1}} \exterioderivative\Xi_{F_{j},S,i} \wedge \omega + (-1)^{n-k+1} \Xi_{F_{j},S,i} \wedge \exterioderivative\omega
        = 0
        .
    \end{align*}
    Consequently 
    \begin{align*}
        &
        \int_{T_{j  }} \exterioderivative\Xi_{T_{j  },F_{j},S,i} \wedge \omega + (-1)^{n-k+1} \Xi_{T_{j  },F_{j},S,i} \wedge \exterioderivative\omega
        \\&\qquad\qquad 
        +
        \int_{T_{j-1}} \exterioderivative\Xi_{T_{j-1},F_{j},S,i} \wedge \omega + (-1)^{n-k+1} \Xi_{T_{j-1},F_{j},S,i} \wedge \exterioderivative\omega
        = 0
        .
    \end{align*}
    Moreover, $o(F_{j},T_{j-1}) = - o(F_{j},T_{j})$,
    because the two $n$-simplices induce opposing orientations on $F$. 
    One derives
    \begin{align*}
        &
        o(F_{j},T_{j})
        \int_{T_{j  }} \exterioderivative\Xi_{T_{j  },F_{j},S,i} \wedge \omega + (-1)^{n-k+1} \Xi_{T_{j  },F_{j},S,i} \wedge \exterioderivative\omega
        \\&\qquad 
        -
        o(F_{j},T_{j-1})
        \int_{T_{j-1}} \exterioderivative\Xi_{T_{j-1},F_{j},S,i} \wedge \omega + (-1)^{n-k+1} \Xi_{T_{j-1},F_{j},S,i} \wedge \exterioderivative\omega
        = 0
        .
    \end{align*}
    We combine our calculations and obtain 
    \begin{align*}
        &
        \int_{F_{j}} \mathring\xi_{F_{j},S,i} \wedge \trace_{T_{j  },F_{j  }} \Pi_{T_{j  }} \omega 
        -
        \int_{F_{j}} \mathring\xi_{F_{j},S,i} \wedge \trace_{T_{j-1},F_{j  }} \Pi_{T_{j-1}} \omega 
        \\&\quad =
        o(F_{j},T_{j})
        \int_{T_{j}} 
        \exterioderivative \Xi_{T_{j},F_{j},S,i} \wedge \left( \Pi_{T_{j}} \omega - \omega \right) 
        + 
        (-1)^{n-k+1} 
        \Xi_{T_{j},F_{j},S,i} \wedge \exterioderivative \left( \Pi_{T_{j}} \omega - \omega \right)
        \\&\quad \quad 
        -
        o(F_{j},T_{j-1})
        \int_{T_{j-1}} 
        \exterioderivative \Xi_{T_{j-1},F_{j},S,i} \wedge \left( \Pi_{T_{j-1}} \omega - \omega \right) 
        + 
        (-1)^{n-k+1} 
        \Xi_{T_{j-1},F_{j},S,i} \wedge \exterioderivative \left( \Pi_{T_{j-1}} \omega - \omega \right)
        .
    \end{align*}
    We use H\"older's inequality again and can summarize
    \begin{align*}
        &
        \left|\int_{F_{j}} \mathring\xi_{F_{j},S,i} \wedge \trace_{T_{j  },F_{j  }} \Pi_{T_{j  }} \omega 
        -
        \int_{F_{j}} \mathring\xi_{F_{j},S,i} \wedge \trace_{T_{j-1},F_{j  }} \Pi_{T_{j-1}} \omega \right|
        \\
        &\qquad\leq
        \| \exterioderivative \Xi_{T_{j},F_{j},S,i} \|_{L^{q}\Lambda^{n-k}(T_{j})}  
        \| \omega - \Pi_{T_{j}} \omega \|_{L^{p}\Lambda^{k}(T_{j})} 
        \\&\qquad\qquad
        + 
        \| \Xi_{T_{j},F_{j},S,i} \|_{L^{q}\Lambda^{n-k-1}(T_{j})}  
        \| \exterioderivative \omega - \exterioderivative \Pi_{T_{j}} \omega \|_{L^{p}\Lambda^{k+1}(T_{j})}  
        \\&\qquad\qquad
        +
        \| \exterioderivative \Xi_{T_{j-1},F_{j},S,i} \|_{L^{q}\Lambda^{n-k}(T_{j-1})}  
        \| \omega - \Pi_{T_{j-1}} \omega \|_{L^{p}\Lambda^{k}(T_{j-1})}  
        \\&\qquad\qquad
        + 
        \| \Xi_{T_{j-1},F_{j},S,i} \|_{L^{q}\Lambda^{n-k-1}(T_{j-1})}  
        \| \exterioderivative \omega - \exterioderivative \Pi_{T_{j-1}} \omega \|_{L^{p}\Lambda^{k+1}(T_{j-1})}  
        \\
        &\qquad\leq
        C_{\Xi} \bigg( 
            h_{S}^{ \frac{n}{q} -n+k }
            \| \omega - \Pi_{T_{j}} \omega \|_{L^{p}\Lambda^{k}(T_{j})} 
            + 
            h_{S}^{ \frac{n}{q} -n+k+1 }
            \| \exterioderivative \omega - \exterioderivative \Pi_{T_{j}} \omega \|_{L^{p}\Lambda^{k+1}(T_{j})}  
            \\&\qquad\qquad\qquad\qquad
            +
            h_{S}^{ \frac{n}{q} -n+k }
            \| \omega - \Pi_{T_{j-1}} \omega \|_{L^{p}\Lambda^{k}(T_{j-1})}  
            + 
            h_{S}^{ \frac{n}{q} -n+k+1 }
            \| \exterioderivative \omega - \exterioderivative \Pi_{T_{j-1}} \omega \|_{L^{p}\Lambda^{k+1}(T_{j-1})}  
        \bigg)
        .
    \end{align*}
    All estimates are in place and we recall that 
    \begin{align*}
        \| \phi_{S,i} \|_{L^{p}\Lambda^{k}(T)} \leq C_{\rm{A}} h_{S}^{ \frac{n}{p} - k  }.
    \end{align*}
    The desired estimate now follows. 
\end{proof}

\begin{corollary}
  Let
  $m \in [0,r+1]$ if $\mathcal P\Lambda^{k}(\mathcal T) = \mathcal P_{r}^{ }\Lambda^{k}(\mathcal T)$
  and let 
  $m \in [0,r  ]$ otherwise. 
  Write $l := \max(0,m-1)$. 
  Then for all $T \in \Delta_{n}(\mathcal T)$
  we have 
  \begin{align*}
    \| \omega - \mathcal J_{\mathcal T}\omega \|_{L^{p}\Lambda^{k}(T)}
    &\leq 
    C_{\mathcal{J},\mathcal U,0}
    \sum_{ \substack{ T' \in \Delta_{n}(\mathcal T) \\ T' \cap T \neq \emptyset } }
    \left( 
        h_{T'}^{m} 
        | \omega |_{W^{m,p}\Lambda^{k  }(T')}
        +
        h_{T'}^{l+1} 
        | \exterioderivative\omega |_{W^{l,p}\Lambda^{k+1}(T')}
    \right)
  \end{align*}
  whenever 
  \begin{align*}
    \omega \in \mathcal W^{p,p}\Lambda^{k}(\Omega) \cap \bigoplus_{ T \in \Delta_{n}(\mathcal T) } W^{m,p}\Lambda^{k}(T)
    .
  \end{align*}
  Here, $C_{\mathcal{J},\mathcal U,0} > 0$ depends only on $n$, $p$, the polynomial degree $r$, and the shape measure of the triangulation.
\end{corollary}

\begin{proof}
 We observe $\exterioderivative \omega_{|T} \in W^{l,p}\Lambda^{k+1}(T)$ 
 for $\omega \in \mathcal W^{p,p}\Lambda^{k}(\Omega)$ and $\omega_{|T} \in W^{m,p}\Lambda^{k}(T)$ with $T \in \Delta_{n}(\mathcal T)$. 
 The results follows by combining Theorem~\ref{prop:scottzhangestimate}, 
 Inequalities \eqref{math:bramblehilbert}--\eqref{math:bestapproximation:exteriorderivative},
 and standard estimates as in previous corollaries. 
\end{proof}

\begin{remark}
 The original Scott-Zhang interpolant was only defined for scalar functions in the Sobolev spaces $W^{s,p}(\Omega)$
 for $p > 1$ and $s > \frac 1 p$. Under those conditions on the parameters $s$ and $p$, traces onto facets are well-defined.
 With regards to scalar functions, 
 we instead constrain ourselves to the case $W^{s,p}(\Omega)$ with $s \geq 1$, 
 as we approach boundary traces only indirectly via an integration by parts formula. 
 That approach generalizes naturally to differential forms. 
 We do not attempt to find an analogue of the low regularity setting on differential forms. 
\end{remark}

\section{Applications} \label{sec:applications}

For the purpose of illustration, we review the results in this article in the setting of three-dimensional vector analysis. This last section centers on applications of the Scott-Zhang interpolant and $L^{2}$ theory. 
Let $\Omega \subseteq \mathbb R^{3}$ be a Lipschitz domain triangulated by a triangulation $\mathcal T$. 
Let $\Gamma \subseteq \partial\Omega$ be a two-dimensional submanifold of the boundary 
triangulated by a subtriangulation $\mathcal U \subset \mathcal T$. 

We let $\mathbf L^{2}(\Omega)$ be the space of square-integrable vector fields over $\Omega$,
and we let $\mathbf H^{m}(\Omega)$ be the space of vector fields with coefficients in $W^{m,2}(\Omega)$. 
We write  
\begin{align*}
    \mathbf H(\curl)
    :=
    \left\{ \mathbf{u} \in \mathbf L^{2}(\Omega) \suchthat \curl \mathbf{u} \in \mathbf L^{2}(\Omega) \right\}
    ,
    \quad 
    \mathbf H(\divergence)
    :=
    \left\{ \mathbf{u} \in \mathbf L^{2}(\Omega) \suchthat \divergence \mathbf{u} \in L^{2}(\Omega) \right\}
    .
\end{align*}
We let $\mathbf{Ned}^{\mathrm{fst}}_{r}(\mathcal T)$ and $\mathbf{Ned}^{\mathrm{snd}}_{r}(\mathcal T)$
be the curl-conforming N\'ed\'elec spaces of first and second kind, respectively,
and $\mathbf B\mathbf D\mathbf M_{r}(\mathcal T)$ and $\mathbf R\mathbf T_{r}(\mathcal T)$
be the divergence-conforming Brezzi-Douglas-Marini space and the Raviart-Thomas space, respectively,
of degree $r$ over $\mathcal T$. These finite element spaces contain the polynomial vector fields up to degree $r$. 

We introduce spaces with boundary conditions along $\Gamma$. 
We write $\mathbf{u} \in \mathbf H(\curl,\Gamma)$ if $\mathbf{u} \in \mathbf H(\curl)$ satisfies 
\begin{align*}
    \int_{\Omega} \langle \curl \mathbf{u}, \phi \rangle = \int_{\Omega} \langle \mathbf{u}, \curl \phi \rangle
\end{align*}
for all vector fields $\phi \in C^{\infty}(\overline\Omega)^{3}$ vanishing near $\partial\Omega \setminus \Gamma$.
Similarly,
we write $\mathbf{u} \in \mathbf H(\divergence,\Gamma)$ if $\mathbf{u} \in \mathbf H(\divergence)$ satisfies 
\begin{align*}
    \int_{\Omega} \langle \divergence \mathbf{u}, \phi \rangle = - \int_{\Omega} \langle \mathbf{u}, \grad \phi \rangle
\end{align*}
for all functions $\phi \in C^{\infty}(\overline\Omega)$ vanishing near $\partial\Omega \setminus \Gamma$.
We set 
\begin{gather*}
    \mathbf{Ned}^{\mathrm{fst}}_{r}(\mathcal T,\mathcal U) := \mathbf H(\curl,\Gamma) \cap \mathbf{Ned}^{\mathrm{fst}}_{r}(\mathcal T),
    \quad 
    \mathbf{Ned}^{\mathrm{snd}}_{r}(\mathcal T,\mathcal U) := \mathbf H(\curl,\Gamma) \cap \mathbf{Ned}^{\mathrm{snd}}_{r}(\mathcal T),
    \\
    \mathbf B\mathbf D\mathbf M_{r}(\mathcal T,\mathcal U) := \mathbf H(\divergence,\Gamma) \cap \mathbf B\mathbf D\mathbf M_{r}(\mathcal T),
    \quad 
    \mathbf R\mathbf T_{r}(\mathcal T,\mathcal U) := \mathbf H(\divergence,\Gamma) \cap \mathbf R\mathbf T_{r}(\mathcal T).
\end{gather*}
These are the finite element spaces with boundary conditions along $\Gamma$.
We can equally define them by setting the degrees of freedom associated 
to simplices in $\mathcal U$ to zero. 

The results in this article include the following theorems as a special case.

\begin{theorem} \label{prop:application:divergence:BDM}
    There exist linear mappings 
    \begin{gather*}
        \mathcal J_{\mathbf B\mathbf D\mathbf M_{r}(\mathcal T,\mathcal U)} : \mathbf H(\divergence,\Gamma) \rightarrow \mathbf B\mathbf D\mathbf M_{r}(\mathcal T,\mathcal U),
    \end{gather*}
    such that for $m \in [0,r+1]$, $l \in [0,r]$, 
    all tetrahedra $T \in \mathcal T$,
    and all $\mathbf{u} \in \mathbf H(\divergence,\Gamma)$ we have 
    \begin{gather*}
        \| \mathbf{u} - \mathcal J_{\mathbf B\mathbf D\mathbf M_{r}(\mathcal T,\mathcal U)} \mathbf{u} \|_{\mathbf L^{2}(T)}
        \leq 
        C        
        \sum_{ \substack{ T' \in \mathcal T \\ \dim(T') = 3 \\ T' \cap T \neq \emptyset } }
        h^{m}_{T}\| \mathbf{u} \|_{\mathbf H^{m}(T')} + h^{l+1}_{T}\| \divergence \mathbf{u} \|_{W^{l,2}(T')}
    \end{gather*}
    whenever the right-hand side is well-defined. 
    Here, the constant $C > 0$ depends only on the polynomial degree $r$ and the shape measure of $\mathcal T$.
\end{theorem}

\begin{theorem} \label{prop:application:divergence:RT}
    There exist linear mappings 
    \begin{gather*}
        \mathcal J_{\mathbf R\mathbf T_{r}(\mathcal T,\mathcal U)}     : \mathbf H(\divergence,\Gamma) \rightarrow     \mathbf R\mathbf T_{r}(\mathcal T,\mathcal U),
    \end{gather*}
    such that for $m \in [0,r+1]$, $l \in [0,r+1]$, 
    all tetrahedra $T \in \mathcal T$,
    and all $\mathbf{u} \in \mathbf H(\divergence,\Gamma)$ we have 
    \begin{gather*}
        \| \mathbf{u} - \mathcal J_{\mathbf R\mathbf T_{r}(\mathcal T,\mathcal U)} \mathbf{u} \|_{\mathbf L^{2}(T)}
        \leq 
        C        
        \sum_{ \substack{ T' \in \mathcal T \\ \dim(T') = 3 \\ T' \cap T \neq \emptyset } }
        h^{m}_{T}\| \mathbf{u} \|_{\mathbf H^{m}(T')} + h^{l+1}_{T}\| \divergence \mathbf{u} \|_{W^{l,2}(T')}
    \end{gather*}
    whenever the right-hand side is well-defined. 
    Here, the constant $C > 0$ depends only on the polynomial degree $r$ and the shape measure of $\mathcal T$.
\end{theorem}

\begin{theorem} \label{prop:application:curl:fst}
    There exist linear mappings 
    \begin{gather*}
        \mathcal J_{\mathbf{Ned}^{\mathrm{fst}}_{r}(\mathcal T,\mathcal U)} : \mathbf H(\curl,\Gamma) \rightarrow \mathbf{Ned}^{\mathrm{fst}}_{r}(\mathcal T,\mathcal U),
    \end{gather*}
    such that for $m \in [0,r+1]$, $l \in [0,r+1]$, 
    all tetrahedra $T \in \mathcal T$,
    and all $\mathbf{u} \in \mathbf H(\curl,\Gamma)$ we have 
    \begin{gather*}
        \| \mathbf{u} - \mathcal J_{\mathbf{Ned}^{\mathrm{fst}}_{r}(\mathcal T,\mathcal U)} \mathbf{u} \|_{\mathbf L^{2}(T)}
        \leq 
        C        
        \sum_{ \substack{ T' \in \mathcal T \\ \dim(T') = 3 \\ T' \cap T \neq \emptyset } }
        h^{m}_{T}\| \mathbf{u} \|_{\mathbf H^{m}(T')} + h^{l+1}_{T}\| \curl \mathbf{u} \|_{\mathbf H^{l}(T')}
    \end{gather*}
    whenever the right-hand side is well-defined. 
    Here, the constant $C > 0$ depends only on the polynomial degree $r$ and the shape measure of $\mathcal T$.
\end{theorem}

\begin{theorem} \label{prop:application:curl:snd}
    There exist linear mappings 
    \begin{gather*}
        \mathcal J_{\mathbf{Ned}^{\mathrm{snd}}_{r}(\mathcal T,\mathcal U)} : \mathbf H(\curl,\Gamma) \rightarrow \mathbf{Ned}^{\mathrm{snd}}_{r}(\mathcal T,\mathcal U),
    \end{gather*}
    such that for $m \in [0,r+1]$, $l \in [0,r]$, 
    all tetrahedra $T \in \mathcal T$,
    and all $\mathbf{u} \in \mathbf H(\curl,\Gamma)$ we have 
    \begin{gather*}
        \| \mathbf{u} - \mathcal J_{\mathbf{Ned}^{\mathrm{snd}}_{r}(\mathcal T,\mathcal U)} \mathbf{u} \|_{\mathbf L^{2}(T)}
        \leq 
        C        
        \sum_{ \substack{ T' \in \mathcal T \\ \dim(T') = 3 \\ T' \cap T \neq \emptyset } }
        h^{m}_{T}\| \mathbf{u} \|_{\mathbf H^{m}(T')} + h^{l+1}_{T}\| \curl \mathbf{u} \|_{\mathbf H^{l}(T')}
    \end{gather*}
    whenever the right-hand side is well-defined. 
    Here, the constant $C > 0$ depends only on the polynomial degree $r$ and the shape measure of $\mathcal T$.
\end{theorem}


\begin{thebibliography}{10}

\bibitem{AFW1}
{\sc D.~N. Arnold, R.~S. Falk, and R.~Winther}, {\em Finite element exterior
  calculus, homological techniques, and applications}, Acta Numerica, 15
  (2006), pp.~1--155.

\bibitem{afwgeodecomp}
\leavevmode\vrule height 2pt depth -1.6pt width 23pt, {\em Geometric
  decompositions and local bases for spaces of finite element differential
  forms}, Computer Methods in Applied Mechanics and Engineering, 198 (2009),
  pp.~1660--1672.

\bibitem{AFW2}
\leavevmode\vrule height 2pt depth -1.6pt width 23pt, {\em Finite element
  exterior calculus: from {Hodge} theory to numerical stability}, Bulletin of
  the American Mathematical Society, 47 (2010), pp.~281--354.

\bibitem{bank2015note}
{\sc R.~E. Bank and H.~Yserentant}, {\em A note on interpolation, best
  approximation, and the saturation property}, Numerische Mathematik, 131
  (2015), pp.~199--203.

\bibitem{bonizzoni2014moment}
{\sc F.~Bonizzoni, A.~Buffa, and F.~Nobile}, {\em Moment equations for the
  mixed formulation of the {Hodge Laplacian} with stochastic loading term}, IMA
  Journal of Numerical Analysis, 34 (2014), pp.~1328--1360.

\bibitem{camacho2015L2}
{\sc F.~Camacho and A.~Demlow}, {\em ${L}^2$ and pointwise a posteriori error
  estimates for {FEM} for elliptic {PDE}s on surfaces}, IMA Journal of
  Numerical Analysis, 35 (2015), pp.~1199--1227.

\bibitem{chaumont2020equivalence}
{\sc T.~Chaumont-Frelet and M.~Vohral{\'\i}k}, {\em Equivalence of local-best
  and global-best approximations in ${H}(\curl)$}, HAL Preprint: hal-02736200,
  (2020).

\bibitem{christiansen2007stability}
{\sc S.~H. Christiansen}, {\em Stability of {H}odge decompositions in finite
  element spaces of differential forms in arbitrary dimension}, Numerische
  Mathematik, 107 (2007), pp.~87--106.

\bibitem{licht2020poincare}
{\sc S.~H. Christiansen and M.~W. Licht}, {\em Poincar{\'e}--{F}riedrichs
  inequalities of complexes of discrete distributional differential forms}, BIT
  Numerical Mathematics, 60 (2020), pp.~345--371.

\bibitem{christiansen2011topics}
{\sc S.~H. Christiansen, H.~Z. Munthe-Kaas, and B.~Owren}, {\em Topics in
  structure-preserving discretization}, Acta Numerica, 20 (2011), p.~1.

\bibitem{christiansen2008smoothed}
{\sc S.~H. Christiansen and R.~Winther}, {\em Smoothed projections in finite
  element exterior calculus}, Mathematics of Computation, 77 (2008),
  pp.~813--829.

\bibitem{clement1975approximation}
{\sc P.~Cl{\'e}ment}, {\em Approximation by finite element functions using
  local regularization}, Revue fran{\c{c}}aise d'automatique, informatique,
  recherche op{\'e}rationnelle. Analyse num{\'e}rique, 9 (1975), pp.~77--84.

\bibitem{demlow2014posteriori}
{\sc A.~Demlow and A.~Hirani}, {\em A posteriori error estimates for finite
  element exterior calculus: The de {Rham} complex}, Foundations of
  Computational Mathematics,  (2014), pp.~1--35.

\bibitem{di2012hitchhiker}
{\sc E.~Di~Nezza, G.~Palatucci, and E.~Valdinoci}, {\em Hitchhiker?s guide to
  the fractional {S}obolev spaces}, Bulletin des sciences math{\'e}matiques,
  136 (2012), pp.~521--573.

\bibitem{dupont1980polynomial}
{\sc T.~Dupont and R.~Scott}, {\em Polynomial approximation of functions in
  {Sobolev} spaces}, Mathematics of Computation, 34 (1980), pp.~441--463.

\bibitem{ern2019equivalence}
{\sc A.~Ern, T.~Gudi, I.~Smears, and M.~Vohral{\'\i}k}, {\em Equivalence of
  local- and global-best approximations, a simple stable local commuting
  projector, and optimal $hp$ approximation estimates in ${H}(\divergence)$},
  arXiv preprint arXiv:1908.08158,  (2019).

\bibitem{ern2016mollification}
{\sc A.~Ern and J.-L. Guermond}, {\em Mollification in strongly {L}ipschitz
  domains with application to continuous and discrete de {R}ham complexes},
  Computational Methods in Applied Mathematics, 16 (2016), pp.~51--75.

\bibitem{ern2017finite}
\leavevmode\vrule height 2pt depth -1.6pt width 23pt, {\em Finite element
  quasi-interpolation and best approximation}, ESAIM: Mathematical Modelling
  and Numerical Analysis, 51 (2017), pp.~1367--1385.

\bibitem{falk2014local}
{\sc R.~Falk and R.~Winther}, {\em Local bounded cochain projections},
  Mathematics of Computation, 83 (2014), pp.~2631--2656.

\bibitem{fernandes1997magnetostatic}
{\sc P.~Fernandes and G.~Gilardi}, {\em Magnetostatic and electrostatic
  problems in inhomogeneous anisotropic media with irregular boundary and mixed
  boundary conditions}, Mathematical Models and Methods in Applied Sciences, 7
  (1997), pp.~957--991.

\bibitem{fortin1977analysis}
{\sc M.~Fortin}, {\em An analysis of the convergence of mixed finite element
  methods}, RAIRO. Analyse num{\'e}rique, 11 (1977), pp.~341--354.

\bibitem{GMM}
{\sc V.~Gol'dshtein, I.~Mitrea, and M.~Mitrea}, {\em {H}odge decompositions
  with mixed boundary conditions and applications to partial differential
  equations on {L}ipschitz manifolds}, Journal of Mathematical Scienes, 172
  (2011), pp.~347--400.

\bibitem{gol1982differential}
{\sc V.~M. Gol'dshtein, V.~I. Kuz'minov, and I.~A. Shvedov}, {\em Differential
  forms on {{L}ipschitz} manifolds}, Siberian Mathematical Journal, 23 (1982),
  pp.~151--161.

\bibitem{gopalakrishnan2011partial}
{\sc J.~Gopalakrishnan and W.~Qiu}, {\em Partial expansion of a {{L}ipschitz}
  domain and some applications}, Frontiers of Mathematics in China,  (2011),
  pp.~1--24.

\bibitem{hiptmair2002finite}
{\sc R.~Hiptmair}, {\em Finite elements in computational electromagnetism},
  Acta Numerica, 11 (2002), pp.~237--339.

\bibitem{iwaniec1999nonlinear}
{\sc T.~Iwaniec, C.~Scott, and B.~Stroffolini}, {\em Nonlinear {H}odge theory
  on manifolds with boundary}, Annali di Matematica pura ed applicata, 177
  (1999), pp.~37--115.

\bibitem{jakab2009regularity}
{\sc T.~Jakab, I.~Mitrea, and M.~Mitrea}, {\em On the regularity of
  differential forms satisfying mixed boundary conditions in a class of
  {{L}ipschitz} domains}, Indiana University Mathematics Journal, 58 (2009),
  pp.~2043--2072.

\bibitem{jochmann1997compactness}
{\sc F.~Jochmann}, {\em A compactness result for vector fields with divergence
  and curl in ${L}^q(\omega)$ involving mixed boundary conditions}, Applicable
  Analysis, 66 (1997), pp.~189--203.

\bibitem{jochmann1999regularity}
\leavevmode\vrule height 2pt depth -1.6pt width 23pt, {\em Regularity of weak
  solutions of {M}axwell's equations with mixed boundary-conditions},
  Mathematical methods in the applied sciences, 22 (1999), pp.~1255--1274.

\bibitem{licht2019mixed}
{\sc M.~Licht}, {\em Smoothed projections and mixed boundary conditions},
  Mathematics of Computation, 88 (2019), pp.~607--635.

\bibitem{licht2019smoothed}
\leavevmode\vrule height 2pt depth -1.6pt width 23pt, {\em Smoothed projections
  over weakly {L}ipschitz domains}, Mathematics of Computation, 88 (2019),
  pp.~179--210.

\bibitem{licht2017complexes}
{\sc M.~W. Licht}, {\em Complexes of discrete distributional differential forms
  and their homology theory}, Foundations of Computational Mathematics, 17
  (2017), pp.~1085--1122.

\bibitem{licht2017priori}
\leavevmode\vrule height 2pt depth -1.6pt width 23pt, {\em On the a priori and
  a posteriori error analysis in finite element exterior calculus}, PhD thesis,
  Dissertation, Department of Mathematics, University of Oslo, Norway, 2017.

\bibitem{mitrea2002traces}
{\sc D.~Mitrea, M.~Mitrea, and M.~Shaw}, {\em Traces of differential forms on
  {{L}ipschitz} domains, the boundary {de Rham} complex, and {Hodge}
  decompositions}, Journal of Functional Analysis, 190 (2002), pp.~339--417.

\bibitem{oswald1993bpx}
{\sc P.~Oswald}, {\em On a {BPX}-preconditioner for {P1} elements}, Computing,
  51 (1993), pp.~125--133.

\bibitem{schoberl2008posteriori}
{\sc J.~Sch{\"o}berl}, {\em A posteriori error estimates for {Maxwell}
  equations}, Mathematics of Computation, 77 (2008), pp.~633--649.

\bibitem{scott1995Lp}
{\sc C.~Scott}, {\em {${L}^p$} theory of differential forms on manifolds},
  Transactions of the American Mathematical Society, 347 (1995),
  pp.~2075--2096.

\bibitem{scott1990finite}
{\sc L.~R. Scott and S.~Zhang}, {\em Finite element interpolation of nonsmooth
  functions satisfying boundary conditions}, Mathematics of Computation, 54
  (1990), pp.~483--493.

\bibitem{slobodeckij1958generalized}
{\sc L.~Slobodeckij}, {\em Generalized {S}obolev spaces and their applications
  to boundary value problems of partial differential equations}, Gos. Ped.
  Inst. Ucep. Zap, 197 (1958), pp.~54--112.

\bibitem{veeser2016approximating}
{\sc A.~Veeser}, {\em Approximating gradients with continuous piecewise
  polynomial functions}, Foundations of Computational Mathematics, 16 (2016),
  pp.~723--750.

\end{thebibliography}
\end{document}